\documentclass[10pt,a4paper]{amsart}
\usepackage[english]{babel}
\usepackage{amsmath,amsthm,mathrsfs,hyperref}
\usepackage{amssymb}
\usepackage{aliascnt}
\usepackage{charter}
\usepackage{comment}
\usepackage{todonotes}
\usepackage{xspace}
\usepackage{nicefrac}
\DeclareMathOperator{\image}{''}
\DeclareMathOperator{\acc}{acc}
\DeclareMathOperator{\dom}{dom}

\DeclareMathOperator{\otp}{otp}
\DeclareMathOperator{\cof}{cf}

\DeclareMathOperator{\len}{len}
\DeclareMathOperator{\Lim}{acc}
\DeclareMathOperator{\crit}{crit}

\DeclareMathOperator{\cf}{cf}

\DeclareMathOperator{\Add}{Add}
\DeclareMathOperator{\Col}{Col}

\DeclareMathOperator{\Ult}{Ult}

\newcommand{\ZFC}{{\rm ZFC}\xspace}

\newcommand{\AC}{{\rm AC}\xspace}

\newcommand{\lusim}[1]{\smash{\underset{\raisebox{1.2pt}[0cm][0cm]{$\sim$}}
{{#1}}}}
\newcommand{\name}{\lusim}

\newtheorem{theorem}{Theorem}
\newaliascnt{example}{theorem}

\aliascntresetthe{example}
\newaliascnt{fact}{theorem}

\aliascntresetthe{fact}
\newaliascnt{corollary}{theorem}
\newtheorem{corollary}[corollary]{Corollary}
\aliascntresetthe{corollary}
\newaliascnt{lemma}{theorem}
\newtheorem{lemma}[lemma]{Lemma}
\aliascntresetthe{lemma}
\newaliascnt{claim}{theorem}
\newtheorem{claim}[lemma]{Claim}
\aliascntresetthe{claim}
\newtheorem{prop}[theorem]{Proposition}
\theoremstyle{definition}
\newaliascnt{definition}{theorem}
\newtheorem{definition}[definition]{Definition}
\aliascntresetthe{definition}
\newtheorem{question}{Question}

\def\calK{\mathcal K}

\def\upr{\upharpoonright}
\def\l{{\langle}}
\def\r{{\rangle}}

\newtheorem*{theorem*}{Theorem}
\newtheorem*{remark}{Remark}
\newtheorem*{notation}{Notation}
\newtheorem*{example*}{Example}

\begin{document}
\title{The variety of projections of a Tree-Prikry forcing}
\author{Tom Benhamou}
\address[Tom Benhamou]{School of Mathematical Sciences, Raymond and Beverly Sackler Faculty of Exact Science, Tel-Aviv University, Ramat Aviv 69978, Israel}
\email[Tom Benhamou]{tombenhamou@tauex.tau.ac.il}

\author{Moti Gitik}\thanks{The work of the second author was partially supported by ISF grant No 1216/18}
\address[Moti gitik]{School of Mathematical Sciences, Raymond and Beverly Sackler Faculty of Exact Science, Tel-Aviv University, Ramat Aviv 69978, Israel}
\email[Moti Gitik]{gitik@post.tau.ac.il}

\author{Yair Hayut}\thanks{The work of the third author was partially supported by the FWF Lise Meitner
grant 2650-N35, and the ISF grant 1967/21}
\address[Yair Hayut]{Einstein Institute of Mathematics, \\
Edmond J.\ Safra Campus, \\
The Hebrew University of Jerusalem \\
Givat Ram.\ Jerusalem, 9190401, Israel}
\email[Yair Hayut]{yair.hayut@mail.huji.ac.il}

\begin{abstract}
We study which $\kappa$-distributive forcing notions of size $\kappa$ can be embedded into  tree Prikry forcing notions with $\kappa$-complete ultrafilters  under various large cardinal assumptions. An alternative formulation - can the filter of dense open subsets of a $\kappa$-distributive forcing notion of size $\kappa$ be extended to a $\kappa$-complete ultrafilter.
\end{abstract}
\maketitle
\date{\today}
\section{introduction}

In this paper we will study possibilities of embedding of $\kappa$-distributive forcing notions of size $\kappa$ into  
Prikry forcings with non-normal ultrafilter or into tree Prikry forcing notions with $\kappa$-complete ultrafilters.
 \\By the result of Kanovei, Koepke and the second author \cite{PrikryCaseGitikKanKoe} every subforcing of the standard Prikry forcing is either trivial or equivalent to the Prikry forcing with the same normal ultrafilter. 
 However, the situation changes drastically if non-normal ultrafilters are used.

Existence of such embedding allows one to iterate distributive forcing notions on different cardinals, see \cite[Section 6.4]{Gitik2010}.

A closely related problem is the possibility of extension of the filter of dense open subsets of a $\kappa$-distributive forcing notion of size $\kappa$
to a $\kappa$-complete ultrafilter, the exact statement is given in theorem \ref{equivalece}.

Clearly, if $\kappa$ is a $\kappa$-compact cardinal, then this follows. Actually more is true---there is a single Prikry type forcing, such that any $\kappa$-distributive forcing notion of size $\kappa$ embeds into it, see
\cite{GitikOnCompactCardinals}.

However,  there are $\kappa$-distributive forcing notion of size $\kappa$ which can be embedded into Prikry forcing notions under much weaker assumptions. Thus, for example,
in \cite{TomMoti}  starting from a measurable cardinal,  a generic extension in which there is a $\kappa$-complete ultrafilter on $\kappa$, $\mathcal{U}$, such that the tree Prikry forcing using $\mathcal{U}$ introduces a Cohen subset of $\kappa$ was constructed.

This paper  investigates  different possibilities which are intermediate between those two extremes. More specifically, let $H$ be a subclass of the $\kappa$-distributive of size $\kappa$ forcings,
we examine the following question:
\vskip 0.2 cm
\begin{center}
    \textit{ Can the dense open filter, $D(\mathbb{Q})$,  of any $\mathbb{Q}\in H$\\ be extended to a $\kappa$-complete ultrafilter?}
\end{center}
\vskip 0.2 cm
Our notations are mostly standard. For general information about Prikry type forcing we refer the reader to \cite{Gitik2010}. For general information about large cardinals we refer the reader to \cite{Kanamori1994}.

Throughout the paper, $p \leq q$ means that $p$ is weaker than $q$.

The structure of the paper is as follows:
\begin{itemize}
    \item Section $2$ is intended to give the reader background and basic definitions which appear in this paper.
    \item The main result of section $3$ is theorem \ref{equivalece}: Let $\mathbb{Q}$ be a $\sigma$-distributive forcing of size $\kappa$. Then $B(\mathbb{Q})$ is a projection of the tree Prikry forcing if and only if $D(\mathbb{Q})$ can be extended to a $\kappa$-complete ultrafilter. Moreover the ultrafilter extending $D(\mathbb{Q})$ must be Rudin-Keisler below the ultrafilters of the tree Prikry forcing.
    \item Section $4$ deals with the class of $\kappa$-strategically closed and ${<}\kappa$-strategically closed forcings. Lemma \ref{CohenProj} establishes that  $\Add(\kappa,1)$ projects onto every $\kappa$-strategically closed forcing of cardinality $\kappa$. Also, we present the forcing that adds a Jensen square (see definition \ref{definition: jensen square}) and  prove that it maximal among all the ${<}\kappa$-strategically closed forcings, this is formulated in Lemma \ref{JenSenProj}.
    \item Section $5$ focuses on upper bounds. In theorem \ref{LesskappaStClUpperBound} we give an upper bound for the claim ``For every ${<}\kappa$-strategically closed forcing of size $\kappa$ $\mathbb{P}$ and every $p\in\mathbb{P}$, $D_p(\mathbb{P})$ can be extended to a $\kappa$-complete ultrafilter".
    In the rest of the section we discuss some weaker version of $\Pi^1_1$-subcompact cardinal which is an upper bound for the claim ``For every $\kappa$-distributive forcing of size $\kappa$ $\mathbb{P}$ and every $p\in\mathbb{P}$, $D_p(\mathbb{P})$ can be extended to a $\kappa$-complete ultrafilter.
    \item Section $6$ is devoted to the forcing $Q$, of shooting a club through the singulars. This forcing is a milestone for the class of ${<}\kappa$-strategically closed forcing of size $\kappa$. In theorem -\ref{thm:lowerbound}, we prove that if we can extend $D(Q)$, then either $\exists \lambda o(\lambda)=\lambda^{++}$ or $o^{\calK}(\kappa)\geq \kappa+\kappa$.

    \item Section $7$ provides a strengthening  of results of section $6$ to $o^{\calK}(\kappa)\geq \kappa^++1$.
    \item Section $8$ defines a class called \textit{masterable forcing}. We show, starting with a measurable, that one can  force that the filter of $D_p(\mathbb{P})$ can be extended to a $\kappa$-complete ultrafilter for every masterable forcing $\mathbb{P}$. In this generic extension we give examples of many important forcing notions which are masterable.
    \item Section $9$ presents forcing notions which do not fall under the examples considered in this paper and present further research directions.
\end{itemize}
 \section{Preliminaries}
Let us recall some basic concepts about forcing notions and Tree Prikry forcing. First, our forcing notions are always separative and have a minimal element. We force upward i.e. $p\leq q$ means that $q\Vdash p\in \dot{G}$. Let us start with the concept of projection:
\begin{definition}
Let $\mathbb{P},\mathbb{Q}$ be forcing notions,  $\pi:\mathbb{P}\rightarrow\mathbb{Q}$ is a projection if
\begin{enumerate}
    \item $\pi$ is order preserving.
    \item $\forall p\in\mathbb{P}\forall \pi(p)\leq q\exists p'\geq p.\pi(p')\geq q$.
 projection.
    \item $Im(\pi)$ is dense in $\mathbb{P}$.
\end{enumerate}
\end{definition}
\begin{definition}
Let $\pi:\mathbb{P}\rightarrow\mathbb{Q}$ be a function
\begin{enumerate}
    \item If $G\subseteq\mathbb{P}$ is $V$-generic, define
$$\pi_*(G)=\{q\in\mathbb{Q}\mid \exists p\in G.q\leq\pi(p)\}$$
\item If $H\subseteq\mathbb{Q}$ is $V$-generic, define the quotient forcing
$$\mathbb{P}/H=\pi^{-1}[H]=\{p\in\mathbb{P}\mid \pi(p)\in H\}$$
With the separative order $p\leq_{\mathbb{P}/H} q$ if an only if for every $q\leq_{\mathbb{P}} r$, $r$ is compatible with $p$.
\end{enumerate}
\end{definition}
\begin{claim}
Let $\mathbb{P},\mathbb{Q}$ be any forcing notions, then:
\begin{enumerate}

\item Let $G\subseteq\mathbb{P}$ be $V$-generic and $\pi:\mathbb{P}\rightarrow\mathbb{Q}$ a projection, then $\pi_*(G)\subseteq\mathbb{Q}$ is $V$-generic
\item Let $H\subseteq\mathbb{Q}$ be $V$-generic and $\pi:\mathbb{P}\rightarrow\mathbb{Q}$ a projection, then if $G\subseteq\mathbb{P}/H$ is $V[H]$-generic, then $G\subseteq\mathbb{P}$ is $V$-generic, moreover, $\pi_*(G)=H$.
\item Let $G\subseteq\mathbb{P}$ be $V$-generic and $\pi:\mathbb{P}\rightarrow\mathbb{Q}$ a projection, then $G\subseteq\mathbb{P}/\pi_*(G)$ is $V[\pi_*(G)]$-generic.
\end{enumerate}
\end{claim}
\begin{definition}
Let $\mathbb{P}$ be a forcing notion, denote by $B(\mathbb{P})$ the complete boolean algebra of regular open sets of $\mathbb{P}$.\end{definition}
 If is known that $\mathbb{P}$ can be identify with a dense subset of $B(\mathbb{P})$ and that $B(\mathbb{P})$ is the unique (up to isomorphism) complete boolean algebra we a dense subset isomorphic to $\mathbb{P}$.
 Moreover, $\mathbb{P}$ and $B(\mathbb{P})$ yield the same generic extensions. Let $G\subseteq\mathbb{P}$ be a $V$-generic filter then $\bar{G}=\{b\in B(\mathbb{P})\mid \exists p\in G.b\leq p\}\subseteq B(\mathbb{P})$ is $V$-generic and if $\bar{G}\subseteq B(\mathbb{P})$ is $V$-generic then $G=\bar{G}\cap \mathbb{P}\subseteq\mathbb{P}$ is $V$-generic.
 For more information about boolean algebras see \cite{ShelahProper} or \cite{AbrahamHandbook}.

 \begin{claim}\label{absoToproj}
Let $\mathbb{P},\mathbb{Q}$ be forcing notions. Then:
\begin{enumerate}
    \item There is a projection  $\pi\colon\mathbb{P}\rightarrow B(\mathbb{Q})$ if and only if there is a $\mathbb{P}$-name $\name{H}$ such that
    for every generic filter $H$ for $\mathbb{Q}$ there is a generic filter $G$ for $\mathbb{P}$ such that $(\name{H})_G=H$.
    \item There is a strong projection  $\pi:\mathbb{P}\rightarrow B(\mathbb{Q})$ iff there is a $\mathbb{P}$-name $\name{H}$ such that
    for every $V$-generic filter $H$ for $\mathbb{Q}$ there is a $V$-generic filter $G$ for $\mathbb{P}$ such that $(\name{H})_G=H$.
\end{enumerate}
\end{claim}
\begin{definition}\label{distributive}
Let $\mathbb{P}$ be a forcing notion and let $\kappa$ be a cardinal. $\mathbb{P}$ is $\kappa$-distributive if for every collection $\mathcal{D}$ of dense open subsets of $\mathbb{P}$, $|\mathcal{D}| < \kappa$, the intersection $\bigcap \mathcal{D}$ is also a dense open subset of $\mathbb{P}$.
\end{definition}
Note that if $\mathbb{P}$ is $\kappa$-distributive then the filter generated by the dense open subsets of $\mathbb{P}$ is $\kappa$-complete.
\begin{notation}
Let $\mathbb{P}$ be a forcing notion. We denote by $\mathcal{D}(\mathbb{P})$ the filter for dense open subsets of $\mathbb{P}$. For $p\in\mathbb{P}$ let $D_p(\mathbb{P})$ be the filter generated by $D(\mathbb{P})$ and the set $\{q\in\mathbb{P}\mid q\geq p\}$.
\end{notation}

Let us define the tree Prikry forcing. Let $\kappa$ be a cardinal,  and let $\vec{\mathcal{U}} = \langle U_{\eta} \mid \eta \in [\kappa]^{<\omega}\rangle$ be a sequence of ultrafilters on $\kappa$, indexed by $[\kappa]^{<\omega}$ which is the set of all finite sequences below $\kappa$. Such that $U_\eta$ concentrate on the set $\kappa\setminus\max(\eta)+1$.

Let us define the forcing $\mathbb{T}_{\vec{\mathcal{U}}}$.
An element in $\mathbb{T}_{\vec{\mathcal{U}}}$ is a pair $\langle s, T\rangle$ where:
\begin{enumerate}
\item $s\in [\kappa]^{<\omega}$.
\item $T\subseteq [\kappa]^{<\omega}$, and for all $t\in T$, $s \trianglelefteq t$.
\item $T$ is $\vec{\mathcal{U}}$-splitting: for all $t \in T$, $\{\nu<\kappa \mid t^\smallfrown \nu \in T\} \in U_{t}$.
\end{enumerate}

For $T\subseteq [\kappa]^{<\omega}$ and $\eta\in T$ we denote $T_\eta = \{s\in[\kappa]^{<\omega} \mid \eta^\smallfrown s\in T\}$.

For $p = \langle s, T\rangle,\ p'=\langle s', T'\rangle \in \mathbb{T}_{\vec{\mathcal{U}}}$, $p' \leq p$ and say that $p$ extends $p'$ if $s\in T'$ and $T \subseteq T'_s$. We denote $p'\leq^* p$ and say that $p$ is a direct extension of $p'$ if $p'\leq p$ and $s = s'$.

We will assume always that each $U_\eta$ is $\kappa$-complete. In this case, the relation $\leq^*$ is $\kappa$-complete.

The following claim is well known \cite[Lemma 3.16]{TomTreePrikry}:
\begin{lemma}[Strong Prikry Lemma]
Let $D \subseteq \mathbb{T}_{\vec{\mathcal{U}}}$ be dense open and let $p = \langle s, T\rangle\in\mathbb{T}_{\vec{\mathcal{U}}}$ be a condition. There is a direct extension of $p\leq^*p^* = \langle s, T^*\rangle$, and a natural number $n$ such that for all $\eta\in T^*$, with $\len \eta = n$, $\langle s^\smallfrown \eta, T^*_\eta\rangle\in D$ and for all $\eta$ such that $\len \eta < n$, $\langle s^\smallfrown \eta, T^*_\eta\rangle\notin D$.
\end{lemma}
When analyzing a tree of measures there is a natural iteration of ultrapowers to consider.
\begin{definition}
Let $\vec{\mathcal{U}}$ be a tree of $\kappa$-complete ultrafilters and $\eta\in[\kappa]^{<\omega}$. For $\vec{\mathcal{U}}$ and $0<n<\omega$, define recursively the $nth$ ultrafilter above $\eta$ derived from $\vec{\mathcal{U}}$, denoted $(\mathcal{U}_{\eta})_n$, to be the following ultrafilter over $[\kappa]^n$: $$(\mathcal{U}_{\eta})_1=\mathcal{U}_{\eta}$$
For $A\subseteq [\kappa]^{n+1}$ define
$$A\in(\mathcal{U}_{\eta})_{n+1}\Longleftrightarrow \{\gamma\in[\kappa]^n \mid A_{\gamma}\in \mathcal{U}_{{\eta}^{\frown}\gamma}\}\in(\mathcal{U}_{\eta})_n  $$
where $$A_{\gamma}=\{\alpha<\kappa\mid \gamma^{\frown}\alpha\in A\}$$
\end{definition}
\begin{definition}\label{iteration}
Let $\vec{\mathcal{U}}$ be a tree of $\kappa$-complete ultrafilters, define recursively the \textit{iteration corresponding} to $\vec{\mathcal{U}}$ above $\eta\in[\kappa]^{<\omega}$. $$j_0=j_{\mathcal{U}_{\eta}}:V\rightarrow M_0\simeq Ult(V,\mathcal{U}_{\eta}), \ \ \delta_0= [id]_{\mathcal{U}_{\eta}}$$
$$j_{n,n+1}:M_n\rightarrow \Ult(M_n,j_{n}(\mathcal{\vec{U}})_{\vec{\eta}^{\frown}\langle\delta_0,\dots,\delta_{n}\rangle})\simeq M_{n+1}$$
 $\delta_{n+1} = [id]_{j_n(\mathcal{\vec{U}})_{\vec{\eta}^{\frown}\langle\delta_0,\dots,\delta_{n}\rangle}}$,$j_{n+1}=j_{n,n+1}\circ j_n$ and $j_{m,n+1}=j_{n,n+1}\circ  j_{m,n}$
\end{definition}
The following theorem can also be found in  \cite{TomTreePrikry}:
\begin{theorem}\label{genericsequence}
Let $M_\omega$ be the $\omega$-th iteration of the iteration corresponding to $\vec{\mathcal{U}}$ above $\vec{\eta}$ i.e. $M_\omega$ is the transitive collapse of the direct limit of the system $\langle M_n,\ j_{n,m}\mid n,m<\omega\rangle$ defined in \ref{iteration}, denote the direct limit embeddings by
$j_{n,\omega}:M_n\rightarrow M_{\omega}$. Then the sequence  $\vec{\eta}^{\frown}\langle\delta_n\mid n<\omega\rangle$ is $M_\omega$-generic for the forcing $j_{\omega}(\mathbb{T}_{\vec{\mathcal{U}}})$.
\end{theorem}
\begin{claim}\label{claim}
For every $A\subseteq[\kappa]^{n}$ $$\langle\delta_0,\dots,\delta_{n-1}\rangle\in j_{n-1}(A)\Longleftrightarrow A\in (\mathcal{U}_{\eta})_n$$
\end{claim}
\begin{proof}
For $n=1$ it is just {\L}o{\'s} theorem $[id]_{\mathcal{U}_\eta}\in j_1(A)\Longleftrightarrow A\in \mathcal{U}_{\eta}=(\mathcal{U}_{\eta})_1$. Assume that the claim holds for $n$, and let $A\subseteq[\kappa]^{n+1}$. Denote by $\vec{\delta}_m=\l\delta_0,\dots,\delta_m\r$, then
$$\vec{\delta}_{n}\in j_{n}(A)\Longleftrightarrow \delta_{n}\in j_{n}(A)_{\vec{\delta}_{n-1}}\Longleftrightarrow j_{n-1}(A)_{\vec{\delta}_{n-1}}\in j_{n-1}(\vec{\mathcal{U}})_{\eta^{\frown}\vec{\delta}_{n-1}}$$ By the definition of $j_{n-1}(A)_{\vec{\delta}_{n-1}}$ and the induction hypothesis we can continue the chain of equivalences
$$\Longleftrightarrow\{\gamma\mid A_{\gamma}\in \mathcal{U}_{\eta^{\frown}\gamma}\}\in (\mathcal{U}_{\eta})_n\Longleftrightarrow A\in (\mathcal{U}_{\eta})_{n+1}$$
\end{proof}

\section{Subforcing of the tree Prikry forcing}
In this section we characterize the $\sigma$-distributive complete subforcings of a tree Prikry forcing. Since no bounded subsets of $\kappa$ are introduced, such a forcing is either trivial or $(\kappa,\kappa)$-centered i.e. it is the union of $\kappa$ many sets $A_i$ for $i<\kappa$ such that each $A_i$ is $\kappa$-directed. Standard arguments show that those forcing notions have to be $\kappa$-distributive.
By a theorem of Gitik (see \cite{GitikOnCompactCardinals}), if $\kappa$ is $\kappa$-compact, then there is a Prikry type forcing which absorbs every $\kappa$-distributive forcing $\mathbb{P}$ of cardinality $\kappa$. A simpler version of this theorem is stated in the following claim:
\begin{claim}\label{ExtensionToProj}
Assume that for every $p\in\mathbb{P}$, we can extend $D_p(\mathbb{P})$ to a $\kappa$-complete ultrafilter $U_p$. Then there is a tree of $\kappa$ complete ultrafilter $$\vec{\mathcal{W}}=\langle W_{\eta}\mid \eta\in[\kappa]^{<\omega}\rangle$$ and a projection $\pi:\mathbb{T}_{\vec{\mathcal{W}}}\rightarrow B(\mathbb{P})$
\end{claim}
\begin{proof}
We would like to turn the ultrafilters $U_p$ to ultrafilters on $\kappa$. For this, we first need to identify $\mathbb{P}$ with $[\kappa]^{<\omega}$ somehow.    We define inductively for every $\eta\in[\kappa]^{<\omega}$ a condition $p_\eta\in \mathbb{P}$. First $p_{\l\r}=0_{\mathbb{P}}$. Assume that $p_\eta$ is defined, and let $\mathbb{P}/p_{\eta}:=\{q\in\mathbb{P}\mid q\geq p_\eta\}$. By assumption $|\mathbb{P}/p_\eta|\leq\kappa$, fix any surjection $f_{\eta}:(\max\{\eta\},\kappa)\rightarrow\mathbb{P}/p_\eta$. Define for every $\alpha\in(\max\{\eta\},\kappa)$, $p_{\eta^{\smallfrown}\alpha}=f_\eta(\alpha)$.

Next we define the ultrafilters $W_{\eta}$ for every $\eta\in[\kappa]^{<\omega}$.  Let $g_\eta:\mathbb{P}/p_\eta\rightarrow (\max(\eta),\kappa)$ be a right inverse of $f_{\eta}$ such that $f_{\eta}\circ g_{\eta}=id_{\mathbb{P}/p_\eta}$. Define $W_{\eta}=g_{\eta*}(U_{p_\eta})$ to be the Rudin-Keisler projection of $U_{p_\eta}$ to $\kappa$ i.e. for $A\subseteq\kappa$:
$$A\in W_{\eta}\Longleftrightarrow g_{\eta}^{-1}[A\setminus \max(\eta)+1]\in U_{p_\eta}$$
 In particular $\vec{\mathcal{W}}:=\l W_\eta\mid \eta\in[\kappa]^{<\omega}\r$ is defined.

 Let us define the following name $$ \name{H}=\{\langle \dot{q},\langle t, T\rangle\rangle\mid q\in\mathbb{P}, \  q\leq p_{t},\langle t, T\rangle\in \mathbb{T}_{\vec{\mathcal{W}}}\}$$
 Then $\Vdash_{\mathbb{T}_{\vec{\mathcal{W}}}}\name{H}$ is $V$-generic for $\mathbb{P}$. Indeed, Let $G\subseteq \mathbb{T}_{\vec{\mathcal{W}}}$ be $V$-generic and let $H=(\name{H})_G$. Assume that $\langle \alpha_n\mid n<\omega\rangle$ is the Prikry sequence produced by $G$, and denote by $p_n=p_{\l\alpha_0,\dots,\alpha_n\r}$, then $$H=\{q\in\mathbb{P}\mid \exists n<\omega\ q\leq p_n\}$$
 
 Note that, $\alpha_{n+1}>\alpha_n$ and by construction $p_{n+1}=f_{\l\alpha_0,\dots,\alpha_n\r}(\alpha_{n+1})\in \mathbb{P}/ p_n$, hence  the $p_n$'s are increasing in the order of $\mathbb{P}$ and  $H$ is a filter. Let $D\subseteq \mathbb{P}$ be dense open. We proceed by a density argument, let $\langle t,T\rangle\in\mathbb{T}_{\vec{\mathcal{W}}}$ then $D$ is dense open above $p_t$ and therefore $D\cap\mathbb{P}/p_t\in U_{p_t}$. It is not hard to check from the definition that $f_t^{-1}[D\cap\mathbb{P}/p_t]\in W_t$. It follows that $succ_T(t)\in W_t$, fix any $\alpha\in f_t^{-1}[D\cap\mathbb{P}/p_t]\cap succ_T(t)$. Consider the condition $\l t^{\smallfrown}\alpha,T_{t^{\smallfrown}\alpha}\r\geq \l t,T\r$. By density, there is $\l s^{\smallfrown}\alpha_{n_0},S\r\in G$ such that $p_{s^{\smallfrown}\alpha_{n_0}}=f_{s}(\alpha_{n_0})\in D$. By the definition of $H$ we conclude that $p_{s^{\smallfrown}\alpha_{n_0}}\in H\cap D$ and $H$ is a $V$-generic filter for $\mathbb{P}$.
 
 Let $\name{H}^*$ be a $\mathbb{T}_{\vec{\mathcal{W}}}$-name for the $B(\mathbb{P})$-generic corresponding to $\name{H}$.
 Now the projection is defined as follows:
 $$\pi(x)=\inf\{b\in B(\mathbb{P})\mid x\Vdash b\in\name{H}^*\}$$
 Clearly $\pi$ is order preserving and dense in $B(\mathbb{P})$. To see that condition $(2)$ holds, is just an abstract argument, take $b\leq \pi(x)$, then $\neg(x\Vdash b^c\in \name{H}^*)$, otherwise $b^c\geq \pi(x)\geq b$. Hence there is an extension $x'\geq x$ such that $x'\Vdash b^c\notin \name{H}^*$, since $\name{H}^*$ is an ultrafilter it follows that $b\in\name{H}^*$. so $\pi(x')\leq b$.
\end{proof}
\begin{remark}
If $D_p(\mathbb{P})$ can be extended to $U_p$ only densely often, then we still get a projection. 
\end{remark}
The following theorem claims that in some sense, this is the only way to get a projection.
\begin{theorem}\label{equivalece}
Let $\mathbb{P}$ be a $\sigma$-distributive forcing of size $\kappa$. The following are equivalent: 
\begin{itemize}
    \item There is a sequence  $\vec{\mathcal{U}}$ of $\kappa$-complete ultrafilters and a projection $\pi\colon \mathbb{T}_{\vec{\mathcal{U}}} \to B(\mathbb{P})$.
    \item For every $p\in \mathbb{P}$, $D_p(\mathbb{P})$ can be extended to a $\kappa$-complete ultrafilter $U_p$.
\end{itemize}
 
\end{theorem}
\begin{proof}
If $D_p(\mathbb{P})$ can be extended to a $\kappa$-complete ultrafilter, use claim \ref{ExtensionToProj}. For the other direction, let $\pi:T_{\vec{\mathcal{U}}}\rightarrow B(\mathbb{P})$ be a projection, denote $\mathbb{T}_{\vec{\mathcal{U}}} = \mathbb{T}$. Without loss of generality, we can assume that $\mathbb{P}=\kappa$, and $\leq_{\mathbb{P}}$ is an order on $\kappa$.

Let $q\in \mathbb{P}$ and  $p = \langle s, T\rangle\in\mathbb{T}$ such that $\pi(p)\geq q$ which exists since $\pi$ is a projection. For every $D\subseteq \mathbb{P}$ dense and open subset above $q$, let $$\bar{D} := \{b\in B(\mathbb{P})\mid \exists a\in D. a\leq b\}$$ Then $\bar{D}\subseteq B(\mathbb{P})$ is dense open, and since $\pi$ is a projection, $D':=\pi^{-1}[\bar{D}]$ is a dense open subset of $\mathbb{T}$ above $p$. By the strong Prikry property, there is a direct extension $p^* = \langle s, T_D\rangle \geq^* p$ and a natural number $n_D<\omega$ such that for all $\eta\in T^*$ such that $\len(\eta)=n_D$, $\pi(\langle s ^\smallfrown \eta, (T_D)_\eta\rangle) \in \bar{D}$, while the projection of any extension of $\langle s, T_D\rangle$ of smaller length is not in $\bar{D}$. We claim that there must be a single $n^*<\omega$ which is an upper bound to the set $$\{n_D \mid D\subseteq \mathbb{P}\text{ dense open above } q\}$$
Otherwise, there is a sequence of dense open subsets $D_m$ above $q$ for which \[\sup_{m<\omega} n_{D_m}=\omega.\]
The forcing $\mathbb{P}$ is $\sigma$-distributive, thus
\[D^*=\bigcap_{m<\omega}D_m\]
is still dense and open above $q$. Consider $n_{D^*}$ and $T_{D^*}$. Any extension $p'$ of length $n_{D^*}$ from $T_{D^*}$ will satisfy $\pi(p') \in \bar{D^*}$ and in particular it will be in $\bar{D_m}$ for all $m$. But let $m$ be so large that $n_{D_m}>n_{D^*}$. This is a contradiction to the definition of $n_{D_m}$.

Let us fix such $n^*$. Next we consider the iterated ultrapower of length $\omega$ using the ultrafilters in $\vec{\mathcal{U}}$.

Let $k = \len s$ (the stem of $p$) and let us denote $s = \langle \delta^*_0, \dots, \delta^*_{k-1}\rangle$. Consider the iteration corresponding to $\vec{\mathcal{U}}$ above $s$, and denote
$\delta^*_{k+n} =\delta_n$.

By theorem \ref{genericsequence}, $\langle \delta^*_n \mid n < \omega\rangle$ is a tree Prikry generic sequence for the forcing $j_{\omega}(\mathbb{T})$ over the model $M_\omega$ and by claim \ref{claim}, this generic filter will contain the condition $j_{\omega}(p)$. Denote by $H_\omega\subseteq j_{\omega}(B(\mathbb{P}))$ the $M_\omega$-generic filter generated by the Prikry sequence in $M_\omega[\l\delta^*_n\mid n<\omega\r]$.

Working in $M_{n^*-1}$, let
\[F = \{x \in j_{n^*-1}(\mathbb{P}) \mid \exists T, j_{n^*-1}(\pi)(\langle \langle\delta^*_0, \dots, \delta^*_{k + n^* - 1}\rangle, T\rangle) \geq x\}.\]
$F \in M_{n*-1}$ and it is a subset of $j_{n^*-1}(\mathbb{P})=j_{n^*-1}(\kappa)$. In particular for every $x\in F$, $j_{n^*-1,\omega}(x)=x$.
Since for every $T$,  $j_{n^*-1,\omega}(\langle \langle\delta_0, \dots, \delta_{k + n^* - 1}\rangle, T\rangle)$ is a member of the generic filter which is generated by the sequence $\langle \delta^*_n \mid n < \omega\rangle$, we conclude that $F\subseteq H_\omega$. Note that $F\in M_{n^*}$, as $M_{n^*-1}$ and $M_{n^*}$ agree on subsets of $j_{n^*-1}(\kappa)$. It follows that $j_{n^*,\omega}(F)=F\in M_\omega$. Thus, there must be a single condition $f\in H_\omega$ forcing $F\subseteq \dot{H_\omega}$. This can be the case only if $f$ is stronger that all elements of $F$. find any $f^*\in\mathbb{P}$ such that $f^*\geq f$. We conclude that for every dense open set $D \subseteq \mathbb{P}$ above $q$, $f\in j_{\omega}(\bar{D})$ and since $D$ is dense open in $\mathbb{P}$, $f^*\in j_{\omega}(D)$.

Let us define:
\[U_q=\{A \subseteq \mathbb{P} \mid f^* \in j_\omega(A)\}\]
$U_q$ is a $\kappa$-complete ultrafilter (since $\crit j_{\omega} = \kappa$) and for all dense open $D \subseteq \mathbb{P}$ above $q$, $D \in U_q$.
\end{proof}
\begin{remark}
In the previous proof we have defined the filter $U_p$ to be $$U_q=\{A\subseteq\mathbb{P}\mid f^*\in j_\omega(A)\}$$
where $f^*\in\mathbb{P}$ was a condition forcing $F\subseteq\Dot{H_\omega}$,  $\Dot{H_\omega}$ being a canonical name for the generic filter of $j_\omega(\mathbb{P})$. In $M_{n^*}$, we will have $F$ bounded in the critical point of $j_{n^*,\omega}$ and therefore $j_{n^*,\omega}(F)=F$. By elementarity of $j_{n^*,\omega}$, there is a condition $q^*\in j_{n^*}(\mathbb{P})$ forcing that $F\subseteq\Dot{H_{n^*}}$ where $\Dot{H_{n^*}}$ is the canonical name for the generic filter of $j_{n^*}(\mathbb{P})$. So we may use $q^*$ in order to define $$U_q=\{A\subseteq\mathbb{P}\mid q^*\in j_{n^*}(A)\}$$
This new definition indicates that if there is a projection from $\mathbb{T}_{\vec{\mathcal{U}}}$ onto $\mathbb{P}$ then there will be a Rudin-Keisler projection of the sequence of ultrafilters $\vec{\mathcal{U}}$ on an ultrafilter extending the filter of dense open subsets of $\mathbb{P}$.
\end{remark}
\begin{definition}
Let $\vec{\mathcal{U}}$ be a tree of $\kappa$-complete ultrafilters and let $W$ be a $\kappa$-complete ultrafilter. We say that $W \leq_{RK} \vec{\mathcal{U}}$ if there is $\vec{\eta} \in \kappa^{<\omega}$ and $n < \omega$ such that  $$W\leq_{RK} (\mathcal{U}_{\vec{\eta}})_n$$.
\end{definition}

\begin{theorem}
Let $\vec{\mathcal{U}}$ be a tree of $\kappa$-complete ultrafilters and let $\mathbb{P}$ be $\sigma$-distributive forcing of cardinality $\kappa$.

If $\mathbb{T}_{\vec{\mathcal{U}}}$ projects onto $B(\mathbb{P})$ then for every $p=\langle\delta_0,\dots,\delta_{k-1},T\rangle\in\mathbb{T}_{\vec{\mathcal{U}}}$ there is a $\kappa$-complete ultrafilter $U_p$ which extends $D_{\pi(p)}(\mathbb{P})$ that contain $p$ and $U_p\leq_{RK} \vec{\mathcal{U}}$.
\end{theorem}
\begin{proof}
The proof is just the continuation of the discussion following the proof of theorem \ref{equivalece}, recall the definition of $U_p$ $$U_p=\{A\subseteq\mathbb{P}\mid q^*\in j_{n^*}(A)\}$$
There exists a function $g:[\kappa]^{n^*}\rightarrow \mathbb{P}$ such that $j_{n^*}(g)(\delta_k,\dots,\delta_{n^*+k})=q^*$. We claim that $U_p=g_*((\mathcal{U}_{\langle\delta_0,\dots,\delta_{k-1}\rangle})_{n^*+1})$. Let $A\subseteq \mathbb{P}$, then
$$A\in U_p\Longleftrightarrow j_{n^*}(g)(\delta_k,\dots,\delta_{n^*+k})\in j_{n^*}(A)\Longleftrightarrow$$ $$\Longleftrightarrow\langle \delta_k,\dots,\delta_{n^*+k}\rangle\in j_{n^*}(g^{-1}[A]) \Longleftrightarrow g^{-1}[A]\in (\mathcal{U}_{\langle\delta_0,\dots,\delta_{k-1}\rangle})_{n^*+1} $$

\end{proof}
\section{Projections of forcings}
The following simple lemma indicates that the difficulty of extending the dense open filter for different forcing notions is related to the existence of projections from other forcing notions.
\begin{lemma}
Let $\pi\colon \mathbb{P}\to\mathbb{Q}$ be a projection of forcing notions and let $\kappa$ be a regular cardinal. If there is a $\kappa$-complete ultrafilter that extends $\mathcal{D}_p(\mathbb{P})$, then there is a $\kappa$-complete ultrafilter that extends $\mathcal{D}_{\pi(p)}(\mathbb{Q})$.
\end{lemma}
\begin{proof}
Let $\mathcal{U}$ be a $\kappa$-complete ultrafilter that extends $\mathcal{D}_p(\mathbb{P})$. Let:
\[\pi^*(\mathcal{U}) = \{A \subseteq \mathbb{Q} \mid \pi^{-1}(A) \in \mathcal{U}\}.\]
It is clear that $\pi^*(\mathcal{U})$ is a $\kappa$ complete ultrafilter. For any dense open set $D\in \mathcal{D}_{\pi(p)}(\mathbb{Q})$, the fact that $\pi$ is a projection ensures that $\pi^{-1}(D)\in D_{p}(\mathbb{P})$. Thus, $D \in \pi^*(\mathcal{U})$.
\end{proof}

For the definition of $\lambda$-strategically closed forcings see \cite{CummingsHand}. 
The proof of the following lemma is a variant of theorem $14.1$ in \cite{CummingsHand}.
\begin{lemma}[Folklore]\label{CohenProj}
Let $\mathbb{P}$ be $\kappa$-strategically closed forcing notion of size $\leq\lambda$. There is a projection from $\Col(\kappa,\lambda)$ onto $\mathbb{P}$.
\end{lemma}
The relevant case for our purpose is the case $\kappa = \lambda$. In this case, $\Col(\kappa,\kappa)\cong\Add(\kappa,1)$. Thus, if $\mathbb{P}$ is a $\kappa$-strategically closed forcing of size $\kappa$ then there is a projection from the Cohen forcing $\Add(\kappa,1)$ onto $B(\mathbb{P})$.

Note that the other direction of lemma \ref{CohenProj} is also true, namely that if there is a projection $\pi:\Add(\kappa,1)\rightarrow B(\mathbb{P})$, then $\mathbb{P}$ must also be $\kappa$-strategically closed.

We conclude that questions about the existence of ultrafilters that extend the dense open filter of $\kappa$-strategically of cardinality $\kappa$ closed forcing notions are equivalent to the same question about the Cohen forcing.

For ${<}\kappa$-strategically closed forcing notions the situation is more involved.

\begin{definition}[Jensen]
Let $\kappa$ be an inaccessible cardinal. A \emph{Jensen Square} on $\kappa$ is a sequence $\langle C_\alpha \mid \alpha \in D\rangle$, such that
\begin{enumerate}
\item $D$ is a club consisting of only limit ordinals.
\item $C_\alpha$ is a club at $\alpha$.
\item $\otp C_\alpha < \alpha$.
\item If $\beta \in \acc C_\alpha$, then $\beta\in D$ and $C_\beta = C_\alpha \cap \beta$.
\end{enumerate}
\end{definition}
Note that if there is a Jensen square on $\kappa$ then $\kappa$ is not a Mahlo cardinal.
The following lemma was proven by Velleman \cite[Theorem 1]{Vell}.
\begin{lemma}\label{Vell}
Let $\kappa$ be an infinite cardinal. If there is a Jensen square on $\kappa$ then every ${<}\kappa$-strategically closed forcing is $\kappa$-strategically closed.
\end{lemma}
There is a standard forcing for adding Jensen square at a cardinal $\kappa$, $\mathbb{S}_\kappa$.
\begin{definition}\label{definition: jensen square}
The conditions of $\mathbb{S}_\kappa$ are pairs of the form $\langle \mathcal C, d\rangle$, such that
\begin{enumerate}
    \item $d \subseteq \kappa$ is closed and bounded (with last element) consisting only of limit ordinals.
    \item $\mathcal{C}$ is a function, $\dom \mathcal{C} = d$.
    \item For every $\alpha \in d$, $\mathcal{C}(\alpha)$ is a club at $\alpha$, $\otp \mathcal{C}(\alpha) < \alpha$.
    \item $\forall \beta \in \acc \mathcal{C}(\alpha)$, $\beta\in d$ and $\mathcal{C}(\beta) = \mathcal{C}(\alpha)\cap \beta$.
\end{enumerate}
For $\langle \mathcal{C}, d\rangle, \langle\mathcal{C}', d'\rangle\in\mathbb{S}_\kappa$, $\langle\mathcal{C}, d\rangle\leq \langle\mathcal{C}', d'\rangle$ if $d=d'\cap(\max(d)+1)$ and $\mathcal{C} = \mathcal{C}' \restriction d$.
\end{definition}
There are many variations of this forcing, some of them can be found in \cite{CumShmSquare}. 
\begin{lemma}[Folklore]\label{JensenStCl}
Let $\kappa$ be a regular cardinal then $\mathbb{S}_\kappa$ is ${<}\kappa$-strategically closed.
\end{lemma}
\begin{proof}
Let us define a strategy $\sigma$ first.
$\sigma(\langle\rangle)=\langle\emptyset,\emptyset\rangle$. Assume that $$\langle \langle C_i,D_i\rangle,\langle E_i,F_i\rangle\mid i<\alpha\rangle$$ is defined and played according to $\sigma$ and let us define $$\sigma(\langle \langle C_i,D_i\rangle,\langle E_i,F_i\rangle\mid i<\alpha\rangle)=\langle C_\alpha,D_\alpha\rangle$$
Denote by $d_i=\max(D_i)$. If $\alpha$ is limit, let $d_\alpha=\sup_{i<\alpha}d_i$. Then $\langle C_\alpha,D_\alpha\rangle$ is defined if and only if $d_\alpha$ is a singular cardinal, in which case $$D_\alpha=(\cup_{i<\alpha}D_i)\cup\{d'_\alpha\}$$  For every $i<\alpha$, $C_\alpha\restriction D_i=C_i$ and $$C_\alpha(d_\alpha)=\bigcup_{i<\alpha}C_i(d_i)$$
If $\alpha=\beta+1$,let $d_\alpha$ be an ordinal of cofinality $\omega$ above
$\max(F_\beta)$ and $\langle x_n\mid n<\omega\rangle$ be a cofinal sequence in $d_\alpha$ such that $x_0>d_{\beta}$. Define $$D_\alpha=F_\beta\cup\{d_\alpha\}$$
Also $C_\alpha\restriction F_\beta=E_\beta$ and
$$C_\alpha(d_\alpha)=C_{\beta}(d_{\beta-1})\cup\{x_n\mid n<\omega\}$$
Obviously, by the inductive construction $C_\alpha$ is coherent. It is not hard to see that $\otp(C_\alpha(d_\alpha)))=\omega\cdot\alpha$.
So the strategy $\sigma_\lambda$ starts by jumping above $\omega\cdot\lambda$, then uses $\sigma$. This guarantees that always $\otp(C_\alpha(d_\alpha))<d_\alpha$ and that $d_\alpha$ for limit $\alpha$ is always singular.
\end{proof}
In general, $|\mathbb{S}_\kappa| = \kappa^{<\kappa}$. Thus, for strongly inaccessible cardinals $\kappa$, $|\mathbb{S}_\kappa| = \kappa$ and it fits to the framework of this paper. For Mahlo cardinal $\kappa$, $\mathbb{S}_\kappa$ is not $\kappa$-strategically closed (otherwise, it would be possible to construct Jensen square sequence in the ground model). Thus, for Mahlo cardinal $\kappa$, $\mathbb{S}_\kappa$ is not isomorphic to a complete subforcing of $\Add(\kappa,1)$.

Let us remark that in models of the form $L[E]$, there is a partial square sequence in the ground model which is defined on all singular cardinals. In those cases the forcing that shoots a club through the singular cardinals clearly adds a Jensen square for $\kappa$.

The following lemma shows that adding a Jensen square to $\kappa$ is maximal between all ${<}\kappa$-strategically closed forcing notions.
\begin{lemma}\label{JenSenProj}
$\mathbb{S}_\kappa \cong \mathbb{S}_\kappa \times \Add(\kappa,1)$. In particular, for every ${<}\kappa$-strategically closed forcing $\mathbb{P}$ of cardinality $\kappa$, there is a projection from $\mathbb{S}_\kappa$ onto $B(\mathbb{P})$.
\end{lemma}
\begin{proof}
Let us define a dense embedding $\pi:\mathbb{S}_\kappa\rightarrow\mathbb{S}_\kappa\times\Add(\kappa,1)$,
for every $A\subseteq\kappa$ and $\alpha<\otp(A)$ let  $A(\alpha)$ be the $\alpha$-th element of $A$ in it's natural enumeration.
Define $E_\omega=\{\alpha+\omega\mid \alpha<\kappa\}$ and for $\alpha\in E_{\omega}$ let $\alpha^-=\max(Lim(\alpha))$ be the maximal limit ordinal below $\alpha$.
Let $\langle\mathcal{C},d\rangle\in\mathbb{S}_{\kappa}$, define $\pi(\langle\mathcal{C},d\rangle)=\langle\langle\mathcal{C}',d\rangle,f\rangle$
such that:
\begin{enumerate}
    \item $dom(f)=\gamma_d$, where $\gamma_d=\otp(d\cap E_{\omega})$.
    \item For $i<\gamma_d$, define $$f(i)=1\Longleftrightarrow ((d\cap E_{\omega})(i))^-+1\in \mathcal{C}((d\cap E_{\omega})(i))$$
    \item $dom(\mathcal{C}')=d$.
    \item $\mathcal{C}'$ is defined inductively. For $\alpha\in d$ let $\beta_\alpha=\max(Lim(\mathcal{C}(\alpha))\cap\alpha)$ and assume that $\mathcal{C}'(\beta)$ is defined coherently for every $\beta<\alpha$.
    \begin{enumerate}
        \item If $\alpha\in d\cap E_{\omega}$, then $\beta_\alpha\leq\alpha^-$ and define $$\mathcal{C}'(\alpha)=\mathcal{C}'(\beta_\alpha)\cup \big( [\beta,\alpha^-]\cap \mathcal{C}(\alpha)\big)\cup\big \{\gamma-1\mid \gamma\in[\alpha^-+2,\alpha)\cap \mathcal{C}(\alpha)\big\}$$
        \item If $\alpha\notin d\cap E_{\omega}$ and $\beta_\alpha=\alpha$ let \[\mathcal{C}'=\bigcup_{\gamma\in \acc(\mathcal{C}(\alpha))\cap\alpha}\mathcal{C}'(\gamma)\]
        \item $\alpha\notin d\cap E_{\omega}$ and $\beta_\alpha<\alpha$ let
        $$\mathcal{C}'(\alpha)=\mathcal{C}'(\beta_\alpha)\cup\big( [\beta_\alpha,\alpha]\cap\mathcal{C}(\alpha)\big)$$
    \end{enumerate}
\end{enumerate}
Let us prove first that $\langle\langle\mathcal{C}',d\rangle,f\rangle\in\mathbb{S}_{\kappa}\times \Add(\kappa,1)$. Obviously, $f\in \Add(\kappa,1)$, it is routine to check that $\langle\mathcal{C}',d\rangle\in \mathbb{S}_{\kappa}$,  show by induction that, $$\otp(\mathcal{C}(\alpha))=\otp(\mathcal{C}'(\alpha)), \ \acc(\mathcal{C}(\alpha))=\acc(\mathcal{C}'(\alpha))$$ and that condition  $(3),(4)$ of definition \ref{definition: jensen square} hold. %Assume this is the case for every $\beta<\alpha$, if $\beta_\alpha<\alpha$ then  $\mathcal{C}'(\alpha)$ is closed as the union of closed sets on disjoint intervals. By definition, $$\otp(\mathcal{C}'(\alpha))=\otp(\mathcal{C}'(\beta_\alpha))+\otp(\mathcal{C}(\alpha)\setminus\beta_\alpha)$$
The induction step use the fact that by removing at most one ordinal below a limit point of the set does not change the order type and does not change limit points of the set. %Therefore $\mathcal{C}(\alpha)\setminus\beta_\alpha$  did not change its order type and has the only limit point $\alpha$ (by the definition of $\beta_\alpha$).
%By induction hypothesis $\otp(\mathcal{C}'(\beta_\alpha))=\otp(\mathcal{C}(\beta_\alpha))$ and $\acc(\mathcal{C}'(\alpha))=\acc(\mathcal{C}(\alpha))$. Thus $\otp(\mathcal{C}'(\alpha))=\otp(\mathcal{C}(\alpha))<\alpha$, coherency follows directly from the induction hypothesis and the definition of $\beta_\alpha$.
%If $\beta_\alpha=\alpha$ then
%by coherency
%\[\mathcal{C}(\alpha)=\bigcup_{\gamma\in \acc(\mathcal{C}(\alpha))\cap\alpha)}\mathcal{C}(\gamma)\]
%Thus by definition of $\mathcal{C}'(\alpha)$ and induction hypothesis
%\[\otp(\mathcal{C}(\alpha))=\sup_{\gamma\in \acc(\mathcal{C}(\alpha))\cap\alpha)}\otp(\mathcal{C}(\gamma))=\sup_{\gamma\in \acc(\mathcal{C}(\alpha))\cap\alpha)}\otp(\mathcal{C}'(\gamma))=\otp(\mathcal{C}'(\alpha))\]
% Moreover, $\gamma\in \acc(\mathcal{C}'(\alpha))$ if and only if it is a limit point of $\gamma\in \acc(\mathcal{C}'(\beta))$ for some $\beta\in \acc(\mathcal{C}(\alpha))\cap\alpha$ by induction hypothesis this happens if and only if $\gamma\in \acc(\mathcal{C}(\beta))$ for some  $\beta\in \acc(\mathcal{C}(\beta))$, and by definition of $\beta_\alpha=\alpha$ this holds if and only if $\gamma\in \acc(\mathcal{C}(\alpha))$. Now that we know that $\acc(\mathcal{C}(\alpha))=\acc(\mathcal{C}'(\alpha))$ and $\otp(\mathcal{C}(\alpha))=\otp(\mathcal{C}'(\alpha))$, conditions $(3),(4)$ follows by the induction hypothesis.

 To see that $\pi\image\mathbb{S}_\kappa$ is dense in $\mathbb{S}_{\kappa}\times\Add(\kappa,1)$, let $$p=\langle\langle\mathcal{N},d\rangle,f\rangle\in \mathbb{S}_{\kappa}\times\Add(\kappa,1)$$ Extend $p$ if necessary to $\langle\langle\mathcal{N}',d'\rangle,f'\rangle$
 so that $dom(f')=\otp(d'\cap E_{\omega})$. This is possible since $f$ can be defined arbitrarily on missing points of its domain and $\gamma_d$ can be increased by extending $\langle \mathcal{N},d\rangle$ at successor steps of $d$ from the set $E_{\omega}$ in a coherent way just as in lemma \ref{JensenStCl}.
 To see that $\langle\langle\mathcal{N}',d'\rangle,f'\rangle\in \pi\image\mathbb{S}_\kappa$,
 define $\langle \mathcal{C},d'\rangle$ recursively. Assume $\alpha\in d'\cap E_{\omega}$ and $\alpha=(d'\cap E_\omega)(i)$. If $f'(i)=0$  define
 $$\mathcal{C}(\alpha)=\mathcal{C}(\beta_\alpha)\cup \big( [\beta,\alpha^-]\cap \mathcal{N}(\alpha)\big)\cup\big \{\gamma+1\mid  \gamma\in(\alpha^-,\alpha)\cap\mathcal{N}(\alpha)\big\}$$

 If $f'(i)=1$ define
 $$\mathcal{C}(\alpha)=\mathcal{C}(\beta_\alpha)\cup \big( [\beta,\alpha^-]\cap \mathcal{N}(\alpha)\big)\cup\big \{\gamma+1\mid  \gamma\in(\alpha^-,\alpha)\cap\mathcal{N}(\alpha)\big\}\cup\{\alpha-\omega+1\}$$

 If $\alpha\notin d'\cap E_\omega$ and $\alpha=\beta_\alpha$ define \[\mathcal{C}(\alpha)=\bigcup_{\gamma\in \acc(\mathcal{N}(\alpha))\cap\alpha}\mathcal{C}(\gamma)\]

Finally if
$\alpha\notin d'\cap E_{\omega}$ and $\beta_\alpha<\alpha$ let
        $$\mathcal{C}(\alpha)=\mathcal{C}(\beta_\alpha)\cup\big( [\beta_\alpha,\alpha]\cap\mathcal{N}(\alpha)\big)$$

 It is routine to check that $\pi$ is an embedding.

For the second part, assume that $\mathbb{P}$ is a ${<}\kappa$-strategically closed forcing, let $G$ be generic for $\mathbb{S}_\kappa$, then $V[G]=V[G'][H]$ where $G'$ is another generic for $\mathbb{S}_\kappa$ and $H$ is $V[G']$-generic for $Add(\kappa,1)$. In $V[G']$, since $\mathbb{S}_\kappa$ is ${<}\kappa$ strategically closed, there are no new plays of $\mathbb{P}$ of length less than $\kappa$, indicating that $\mathbb{P}$ stays ${<}\kappa$-strategically closed in $V[G']$. Since in $V[G']$ there is a square sequence, use \ref{Vell} to conclude that $\mathbb{P}$ is $\kappa$-strategically closed in $V[G']$. Thus by \ref{CohenProj}, there is $\pi:\Add(\kappa,1)\rightarrow \mathbb{P}\in V[G']$ a projection. Let us turn this projection to a projection in $V$ of $\mathbb{S}_\kappa\times\Add(\kappa,1)$. Let $\tilde{\pi}$ be a $\mathbb{S}_\kappa$-name such that $\Vdash_{\mathbb{S}_\kappa}\tilde{\pi}:\Add(\kappa,1)\rightarrow B(\mathbb{P})$ is a projection. Consider the set $$D=\{\l p,q\r\in\mathbb{S}_\kappa\times\Add(\kappa,1)\mid \exists a\in \mathbb{P}. p\Vdash \tilde{\pi}(q)=a\}$$
It is dense in $\mathbb{S}_\kappa\times\Add(\kappa,1)$. For every $\l p,q\r\in D$, define $\pi_*(\l p,q\r)=a$ for the unique $a\in \mathbb{P}$, such that $p\Vdash \tilde{\pi}(q)=a$. It is a straightforward verification to see that $\pi_*:D\rightarrow \mathbb{P}$ is a projection. 
\end{proof}
The following lemma shows that $\mathbb{S}_\kappa$ is not maximal among the $\kappa$-distributive forcing notions. For a fat stationary set $S\subseteq \kappa$, let $Club(S)$ be the forcing that shoots a club through $S$ using closed and bounded conditions. By \cite{AbrahamShelah1994}, if $\kappa^{<\kappa} = \kappa$, then $Club(S)$ is $\kappa$-distributive if and only if $S$ is fat stationary set.
\begin{lemma}
Let $S \subseteq T \subseteq \kappa$ be fat stationary sets.
If the set of all $\alpha\in T \setminus S$ such that $T\cap\alpha$ contains a club at $\alpha$  is stationary, then $T\setminus S$ stays stationary in $V^{Club(T)}$ and in particular there is no projection from $Club(T)$ to $Club(S)$.
\end{lemma}
\begin{remark}
After adding a single Cohen set to $\kappa$, there is a partition of $\kappa$ into $\kappa$ many disjoint fat stationary sets. Thus, the structure of the $\kappa$-distributive forcing notions of size $\kappa$ might be complicated in general, even when $\kappa$ is a large cardinal.
\end{remark}
\begin{proof}
Let $C\subseteq T$ be a $V$-generic club for $Club(T)$. Assume that $S$ is not stationary in $V[C]$ and let $\name{B}$ be a name such that some $p\in Club(T)$ forces that $\name{B}$ is a club disjoint from $T\setminus S$. Let $\langle M_i\mid i<\kappa\rangle$ be an increasing and continuous chain of elementary substructures of $H(\theta)$ for some large enough $\theta$ such that:
\begin{enumerate}
    \item $p,\name{B},S,T,Club(T)\in M_0$.
    \item $|M_i|<\kappa$.
    \item $x_i:=M_i\cap\kappa\in \kappa$.
    \item and ${}^{x_i}M_i\subseteq M_{i+1}$.
\end{enumerate}
 Consider the club $\{\alpha\mid x_{\alpha}=\alpha\}$. There is $\alpha<\kappa$ such that $x_\alpha=\alpha\in T\setminus S$ and there is a closed unbounded set $D\subseteq T\cap \alpha$. Let us construct an increasing sequence of conditions $\langle p_i\mid i<\theta\rangle$ such that:
 \begin{enumerate}
     \item $p_0=p$ and $p_i\in M_{\alpha}$.
     \item $[p_{i+1}\setminus\max(p_i)]\cap D\neq\emptyset$.
     \item there is $\max(p_i)<y_i\in M_{\alpha}$ such that $p_{i+1}\Vdash y_i\in\name{B}$.
 \end{enumerate} that $\langle p_i\mid i<j\rangle$ is defined and let $\eta=\sup(\max(p_i)\mid i<j)$. If $j$ is limit and $\eta=\alpha$, define $\theta=j$ and stop. Otherwise, there is $r<\alpha$ such that $\l p_i\mid i<j\r\subseteq M_{r}$ thus in $M_{j+1}$. By closure of $D$, $\eta\in D$, hence it is safe to define $$p_j=\cup_{i<j}p_i\cup\{\eta\}\subseteq T$$
  which is definable in $M_{\alpha}$. For the successor step, assume $p_i\in M_{\alpha}$ is define. Work inside $M_{\alpha}$ and let $p'_{i+1}$ be a condition deciding a value $y_{i}\in\name{B}$ above $\max(p_i)$. Since $D$ is unbounded, there is $z\in D\setminus \max(p'_{i+1})$ then $p_{i+1}=p'_{i+1}\cup\{z\}\in M_{\alpha}$ is as wanted.
  Finally, $\cup_{i<\theta}p_i\cup\{\alpha\}\in Club(T)$ must force that $\alpha\in \name{B}\cap(T\setminus S)$ which is a contradiction.
\end{proof}
\section{Implications}
In this section we will show that certain large cardinals weaker than $\kappa$-compacts already imply  an existence of a $\kappa$-complete ultrafilter extending the filters $\mathcal{D}_p(\mathbb{P})$.

Let us deal first with ${<}\kappa$-strategically closed forcing notion of size $\kappa$.

Recall that a cardinal $\kappa$ is called \emph{superstrong} if and only if there is an elementary embedding
$j\colon V \to M$ such that $\crit(j)=\kappa$ and $V_{j(\kappa)}\subseteq M$.

While $j(\kappa)$ is always a strong limit cardinal, and inaccessible in $M$, it need not be regular in $V$. Actually, $\kappa^{+} \leq \cof(j(\kappa)) \leq 2^\kappa$, for the first such cardinal, see \cite{Perlmutter2015}.
%\\However, $j(\kappa)$ is a cardinal in $V$ and is a strong limit there, since $V_{j(\kappa)}\subseteq M$ and $j(\kappa)$ is a measurable in $M$.
However, if $\cof(j(\kappa))>\lambda$, then ${}^\lambda j(\kappa)\subseteq M$.
%\\Let use such form of superstrongnees to show the following:

%The following theorem enables us to obtain a $\kappa$-complete ultrafilter extending the dense open filter of certain forcing notions.
\begin{theorem}\label{LesskappaStClUpperBound}
Suppose that there is an elementary embedding $j\colon V \to M$ such that $\crit j = \kappa$ and ${}^{2^\kappa}j(\kappa)\subseteq M$. Let $\mathbb{P}$ be a $<\kappa$-strategically closed forcing notion of size $\kappa$.

Then for every $p\in\mathbb{P}$ there is a $\kappa$-complete ultrafilter that extends $\mathcal{D}_p(\mathbb{P})$.
\end{theorem}
\begin{proof}

  Assume without loss of generality that $\mathbb{P}=\kappa$. Fix $p\in \mathbb{P}$. Denote $2^\kappa$ by $\lambda$.
  \\Clearly, $\lambda<j(\kappa)$, since $\mathcal{P}(\kappa)\subseteq M$ and $j(\kappa)$ is a measurable in $M$.
   \\Let $\langle D_\alpha \mid \alpha < \lambda\rangle$ be an enumeration of all subsets of $\mathbb{P}$ in $V$ which are dense above $p$ and open.
In $M$, let $\Sigma$ be a winning strategy for the 
game on $j(\mathbb{P})$ of length $\lambda + 1$.
Such $\Sigma$ exists since $j(\mathbb{P})$ is ${<}j(\kappa)$-strategically closed.

Let us pick by induction a sequence of conditions $p_\alpha \in j(\mathbb{P})=j(\kappa)$, $\alpha < \lambda$, such that $\forall \alpha < \beta$, $p_\alpha \leq p_\beta$ and $p_{\alpha + 1} \in j(D_\alpha)$. 

First, let $p_0=p$. Each condition $p_{\alpha}$ is played by Player I according to $\Sigma$ and $q_\alpha$ is played by Player II, to be a condition stronger than $p_\alpha$ in $j(D_\alpha)$. 

While the sequence $\langle j(D_\alpha) \mid \alpha < \lambda\rangle$ might not be in $M$, 
the sequence $\langle p_\alpha, q_\alpha \mid \alpha < \lambda\rangle$ is in $M$, since ${}^{\lambda}j(\kappa)\subseteq M$, and it is a play which is played according to the strategy $\Sigma$. Therefore, it has an upper bound $\tilde{p}$ which is stronger than all the conditions $p_\alpha, q_\alpha$, $\alpha < \lambda$. By construction, $\tilde{p}\in \bigcap_{\alpha < \kappa^+} j(D_\alpha)$.

Finally,
$$U=\{X\subseteq \kappa \mid \tilde{p}\in j(X)\}$$ will be as desired.
\end{proof}

 The assumption of the theorem cannot be optimal since 
 $V_{\kappa + 2} \subseteq M$, and thus  it is true in $M$ as well that for every ${<}\kappa$-strategically closed forcing notion $\mathbb{P}$, there is a $\kappa$-complete ultrafilter that extends its dense open filter. Thus, by reflection, the conclusion holds for many cardinals below $\kappa$ as well.

Next, we turn to the class of $\kappa$-distributive forcings. %for which will be compered with subcompact cardinals.
The upper bound in this case is a $1$-extendable cardinal:
\begin{definition}
A cardinal $\kappa$ is called $1$-extendible if there is a non-trivial elementary embedding $j:V_{\kappa+1}\rightarrow V_{\lambda+1}$ such that $crit(j)=\kappa$.
\end{definition}

\begin{prop}
If $\kappa$ is $1$-extendible then for every $\kappa$-distributive forcing $\mathbb{P}$ of size $\kappa$ and every $p\in\mathbb{P}$, the filter $D_p(\mathbb{P})$ can be extended to a $\kappa$-complete ultrafilter.
\end{prop}
\begin{proof}
 Code $\l\mathbb{P},\leq_{\mathbb{P}}\r$ and an order of $\kappa$. Thus we can assume without loss of generality that $\l \mathbb{P},\leq_{\mathbb{P}}\r\in V_{\kappa+1}$. Recall that $D_p(\mathbb{P}):=\{D\subseteq \mathbb{P}\mid D\text{ is dense open above }p\}$. 
 Then $D_p(\mathbb{P})\subseteq V_{\kappa+1}$ and it is definable from $\l\mathbb{P},\leq_{\mathbb{P}}\r$. Since $j(\kappa)=\lambda$, and $V_{\kappa+1}\models \kappa$ is inaccessible cardinal, by elementarity $V_\lambda\models \lambda$ is an inaccessible cardinals,
 and therefore $\lambda$ is inaccessible cardinal in $V$. In particular, $V_\lambda$ is closed under $<\lambda$ sequences and the set $j''D_p(\mathbb{P})=\{j(D)\mid D\in D_p(\mathbb{P})\}\in V_{\lambda+1}$. By elementarity
 $j(\mathbb{P})$ is $\lambda$-distributive, and since $|j''D_p(\mathbb{P})|=|D_p(\mathbb{P})|=2^k<\lambda$, we conclude that $\cap_{D\in D_p(\mathbb{P})}j(D)$ is dense open in $j(\mathbb{P})$.
 In particular it is non empty and we can fix any $p^*\in \cap_{D\in D_p(\mathbb{P})}j(D)$. Now in $V$ we can define
 $$F=\{X\subseteq\mathbb{P}\mid p^*\in j(X)\}$$
 As in previous arguments, this $F$ is a $\kappa$-complete ultrafilter extending $D_p(\mathbb{P})$
\end{proof}

We deal here with the following weakening of $\kappa$-compactness:
\begin{center}
\emph{
For every $\kappa$-distributive forcing notion of cardinality $\kappa$, the filter of its dense open subsets can be extended to a $\kappa$-complete ultrafilter.}

\end{center}

In this context, the major difference between $\kappa$ being $\kappa$-compact and $\kappa$ being $1$-extendible is that we do not need to extend \textit{every} $\kappa$-complete filter on $\kappa$. For our purposes we are only interested in extending filters which are definable using a parameter which is a subset of $\kappa$. This distinction leads to the realm of subcompact cardinals. Subcompact cardinals were defined by R.\ Jensen:

 \begin{definition} A cardinal
$\kappa$ is called \emph{subcompact} if
 for every $A\subseteq H(\kappa^+)$,
there are  $\rho <\kappa$,  $B\subseteq H(\rho^+)$
 and an elementary embedding
$$j\colon \l H(\rho^+),\in, B \r \to \l H(\kappa^+),\in, A \r$$
with critical point $\rho$, such that $j(\rho) = \kappa$.
\end{definition}

The following strengthening  was introduced by I.\ Neeman and J.\ Steel \cite{NeemanSteelSubcompact}:

\begin{definition}\label{def-pi1-1}
$\kappa$ is called \emph{$\Pi_1^1$-subcompact} if
 for every $A\subseteq H(\kappa^+)$ and for every $\Pi_1^1$-statement $\Phi$,
 if $\l H(\kappa^+),\in, A \r\models \Phi$  then there are  $\rho <\kappa$ and $B\subseteq H(\rho^+)$
 such that
$\l H(\rho^+),\in, B \r\models \Phi$ and there is an elementary embedding
$$j:\l H(\rho^+),\in, B \r \to \l H(\kappa^+),\in, A \r$$
with critical point $\rho$, such that $j(\rho) = \kappa$.
\end{definition}
The third author showed in \cite{YairSquare} the following:

\begin{theorem}\label{thm-pi1-1}
If $\kappa$ is $\Pi_1^1$-subcompact, then it is a $\kappa$-compact cardinal i.e.\ every $\kappa$-complete filter over $\kappa$ extends to a $\kappa$-complete ultrafilter.

On the other hand, if $\kappa$ is $\kappa$-compact then $\square(\kappa)$ and $\square(\kappa^+)$ fails.
\end{theorem}

The failure of square at two consecutive cardinals seem to have very high consistency strength, which made the conjecture that $\kappa$-compactness is equiconsistent with $\Pi_1^1$-subcompactness plausible. However, a recent work of
Larson and Sargsyan, \cite{LarsonSargsyan2021}, casts doubt on
this heuristic by showing that the consistency strength of the failure of two consecutive squares at $\omega_3$ and $\omega_4$
is below a Woodin limit of Woodin cardinals.

Let us start with the following observation:
\begin{prop}
Let $\kappa$ be a subcompact cardinal such the filter $F_{\mathbb Q}$ of dense open subsets of $\mathbb Q$ extends to a $\kappa$-complete ultrafilter over $\kappa$, for every
$\kappa$-distributive poset $\mathbb Q$ of size $\kappa$.

Then $\kappa$ is a limit of cardinals with the same extension property.
\end{prop}
\begin{proof}
Clearly it is enough to deal with posets which are partial orders on the set $\kappa$. 

For every such $\mathbb Q$, fix a $\kappa$-complete ultrafilter $ F^*_{\mathbb Q}$ over $\kappa$ which extends $F_{\mathbb Q}$.
%By $\GCH$, it can be codded into a subset of $\kappa^+$. 
The ultrafilter $F^*_{\mathbb{Q}}$ is a subset of $P(\kappa) \subseteq H(\kappa^{+})$. Thus, one can code the set:
\[\{\l \mathbb Q, F^*_{\mathbb Q} \r \mid \mathbb Q\subseteq \kappa, \text{ is } \kappa\text{-distributive}\}\]
by a subset $A$ of $H(\kappa^+)$ (for example, we can set $A = \bigcup_{\mathbb{Q}} \{\mathbb{Q}\} \times F^*_{\mathbb{Q}}$).

Now by the definition of subcompactness, with parameter $A$,
there are $\rho <\kappa$ and $B\subseteq \rho^+$
%$\l H(\rho^+),\in, B \r\models \Phi$ and
and an elementary embedding
$$j\colon \l H(\rho^+),\in, B \r \to \l H(\kappa^+),\in, A \r$$
with critical point $\rho$, such that $j(\rho) = \kappa$.

Since the set of all $\rho$-distributive posets is
definable in $H(\rho^+)$, and $j$ is elementary, the set $B$ is a code of the set:
\[\{ \l \mathbb Q, F^*_{\mathbb Q} \r \mid \mathbb{Q} \subseteq \rho \text{ is }\rho\text{-distributive}\},\]
and in particular for every $\rho$-distributive poset of size $\rho$, $\mathbb Q$, there is a $\rho$-complete filter extending $F_{\mathbb{Q}}$. 
\end{proof}
Let us consider now the lightface version of Definition \ref{def-pi1-1}:
\begin{definition} A cardinal
$\kappa$ is called \emph{lightface  $\Pi_1^1$-subcompact} if
\\for every $\Pi_1^1$-statement $\Phi$,
 if $\l H(\kappa^+),\in\r\models \Phi$  then there is $\rho <\kappa$
 such that
$\l H(\rho^+),\in\r\models \Phi$ and there is an elementary embedding
$$j:\l H(\rho^+),\in \r \to \l H(\kappa^+),\in \r$$
with critical point $\rho$, such that $j(\rho) = \kappa$.
\end{definition}

The definition does not allow us to add parameters from $H(\kappa^+)$ to the formula $\Phi$, and thus this large cardinal property is witnessed by a countable set of elementary embeddings. 

The next proposition is similar to \ref{thm-pi1-1}:
\begin{prop}
Let $\kappa$ be a lightface $\Pi^1_1$-subcompact. Then every $\kappa$-distributive forcing $\mathbb{P}$ of size $\kappa$, and every $p\in\mathbb{P}$, the filter $D_p(\mathbb{P})$ can be extended to a $\kappa$-complete ultrafilter.
\end{prop}
\begin{proof}
Assume otherwise, and let $\Phi$ be the statement that there is  $\mathbb{P}=\langle \kappa,\leq_{\mathbb{P}}\rangle\in H(\kappa^+)$ and no ulrafilter extending $D_p(\mathbb{P})$. $\Phi$ is of the form $$\underset{First \ order}{\underbrace{\exists\mathbb{P}}}\underset{Second \ order}{\underbrace{\forall U}} \underset{First \  order}{\underbrace{\mu(\mathbb{P},U)}}$$
Using $\AC$, such a formula can be expressed as a $\Pi^1_1$ formula \cite[P. 153, Lemma 7.2]{Drake}.
Note that $\Phi$ is defined with no parameters, hence by a lightface $\Pi^1_1$-subcompactness of $\kappa$, there is $\rho<\kappa$ an elementary embedding
$$j:H(\rho^+)\rightarrow H(\kappa^+)$$
with critical point $\rho$ such that $j(\rho)=\kappa$,
such that $H(\rho^+)\models\Phi$.

 Therefore there is $\mathbb{P}_\rho$ which is a counterexample of a forcing of size $\rho$ which is $\rho$-distributive such that there is no $\kappa$-complete filter extending the filter $D_p(\mathbb{P}_\rho)$ for some $p\in\mathbb{P}_{\rho}$. Let us enumerate all dense open subsets of $\mathbb{P}_{\rho}$ above $p$ by $\langle D_i\mid i<2^\rho\rangle$. The sequence $\langle j(D_i)\mid i<2^\rho\rangle$ is in $H(\kappa^+)$, since $2^\rho < \kappa$. By elementarity, $j(\mathbb{P}_{\rho})$ is $j(\rho)$- distributive and therefore $\bigcap_{i<\rho^+}j(D_i)\neq\emptyset$, so let $x$ be an element in the intersection. Then
$$\{X\subseteq\mathbb{P}_{\rho}\mid x\in j(X)\}$$
is a $\rho$-complete ultrafilter extending $D_p(\mathbb{P}_{\rho})$ --- a contradiction to the choice of $\mathbb{P}_\rho$.
\end{proof}
To see that the notion of a lightface $\Pi^1_1$-subcompact is strictly weaker than $\Pi^1_1$-subcompact we have the following proposition:
\begin{prop}
Let $\kappa$ be  $\Pi^1_1$-subcompact. 
Then $\kappa$ is a limit
of lightface $\Pi^1_1$-subcompact cardinals.
\end{prop}
\begin{proof}
 Suppose that $\kappa$ is a $\Pi^1_1$-subcompact cardinal.
Let $\Phi$ be a $\Pi^1_1$ statement (with no parameters). If $\Phi$ holds in $H(\kappa^{+})$ let 
$$j_{\Phi}\colon\langle H(\rho_\Phi^+),\in\rangle\rightarrow \langle H(\kappa^+),\in\rangle$$ witness the reflection. Otherwise, let $B_\Phi$ be a subset of $H(\kappa^{+})$ witness the negation of $\Phi$. Let $T\subseteq \omega$ be the set of all G{\"o}del number of true $\Pi^1_1$-formulas in $H(\kappa^{+})$, and let $k_\Phi$ be the G{\"o}del number of $\Phi$.

For each $\Phi$ such that $k_\Phi \in T$, $j_\Phi\subseteq H(\rho_\Phi^+)\times H(\kappa^+)\subseteq H(\kappa^+)$.
 There are only countably many such formulas $\Phi$, and thus we can code all those elementary embeddings as a single subset $A_T\subseteq H(\kappa^+)$\footnote{indeed, $A_T$ is an element of $H(\kappa^{+})$.}. 
 
 Similarly, we can gather all the sets $B_\Phi$ for  $\Phi$ such that $k_\Phi \notin T$, into a single subset of $H(\kappa^{+})$, $B$. Since the truth values of first order formulas is $\Delta_1^1$, we can take a $\Pi^1_1$-formula $\Lambda$ with parameter $k$, using the predicate $B$ such that $\Lambda(k, B)$ if and only if $k$ is the G{\"o}del number of a $\Pi^1_1$-formula $\Phi$ and $B_k$ is a counterexample doe $\Phi$. 
 
 Let $A$ be a set coding $T$, $A_T$ and $B$.
 
 There is a universal $\Pi^1_1$-formula $\Psi(y)$ where $y$ is a first order free variable such that for every regular cardinal
$\beta$, every $\Pi^1_1$ statement $\phi$, \[H(\beta)\models \phi\Longleftrightarrow H(\beta)\models \Psi(k)\] 
for some natural number $k$ which is the G{\"o}del numbering of formulas \cite[p. 272, Lemma 1.9]{Drake}.
%For any $\Pi^1_1$-statement $\alpha$ such that $H(\kappa^+)\models \alpha$, pick such natural number $k_\alpha$ we let $A$ be the set of triples $\langle k_{\Phi},x,j_{\Phi}(x)\rangle\in H(\kappa^+)$. 
In the language of the model 
$\l H(\kappa^+),\in, A\r$ we can formulate the statement $\alpha(A)$ "For every $\Pi^1_1$-statement $\Phi$, $k_\Phi\in T$ implies $\Psi(k_\Phi)$ and $k_\alpha \notin T$ implies $\Lambda(B, k)$".
Now, $\l H(\kappa^+),\in, A\r\models \alpha(A)$,
 apply $\Pi^1_1$-subcompactness to $A$,
there are $\rho < \kappa$,  $B \subseteq H(\rho^+)$ and
 an elementary embedding:
$$j\colon\langle H(\rho^+),\in,B\rangle\rightarrow \langle H(\kappa^+),\in,A\rangle$$
such that $\crit(j)=\rho$, $j(\rho)=\kappa$ and $\langle H(\rho^+),\in,B\r\models \alpha(B)$.
Let us show that $\rho$ is lightface $\Pi^1_1$-subcompact and we will be done. Let $\zeta$ be a $\Pi^1_1$-statement, such that $\langle H(\rho^+),\in\rangle\models \zeta$, then $k_\zeta$ is coded in $B_1$ and by elementarity of $j$ also in $A_1$, hence  $\langle H(\kappa^+),\in\rangle\models\zeta$.
So there is an embedding $j_{\zeta}$ coded by $A$.
In particular for $\rho_\zeta<j(\rho)$, $$\l H(\kappa^+),\in,A\r\models  \{x\mid \l k_\zeta,x,j_\zeta(x)\r\in A\}=H(\rho_\zeta^+)$$ by elementarity of $j$, there is $\rho'_\zeta<\rho$ such that $$\langle H(\rho^+),\in,B\rangle\models\{x\mid \l k_\zeta,x,j_\zeta(x)\r\in B\}=H(\rho_\zeta^{'+}).$$
It must be that $\rho'_\zeta=j(\rho_\zeta')=\rho_\zeta$, since the critical point is $\rho$. For every $x\in H(\rho_\zeta^+)$ there is a unique $y$ such that $\langle k_\zeta,x,y\r\in B$, define $i_\zeta(x)=y$. So
$$i_\zeta:\l H(\rho_\zeta^+),\in\r\rightarrow\l H(\rho^+),\in \r$$
We claim that $i_\zeta$ is elementary, and that $i_\zeta(\rho_\zeta)=\rho$. This will follow after we show that $j\circ i_\zeta=j_\zeta$. Indeed, $\langle k_\zeta, x,i_{\zeta}(x)\rangle\in B$ and be elementarity $\l k_\zeta, j(x),j(i_{\zeta}(x))\r\in A$ but $j(x)=x$ since $\rho_\zeta^+<\rho$ and therefore $\l k_\zeta ,x,j(i_{\zeta}(x))\r=\l k_{\zeta}, x, j_{\zeta}(x)\r$ in particular $j(i_\zeta(x))=j_\zeta(x)$.
\end{proof}

\section{Lower bound}

In this section we deal with the forcing notion for shooting a club through the stationary set of singular ordinals below $\kappa$, i.e.\
 $$Q=\{a \subseteq \kappa \mid |a|<\kappa, a \text{ is closed and each member of } a \text{ is singular}\} $$
 ordered by end-extension.

This forcing is $<\kappa$-strategically closed. In our framework, $\kappa$ is strongly inaccessible and thus this forcing is of cardinality $\kappa$. 

Our aim will be to show the following:

\begin{theorem}\label{thm:lowerbound}
 Let us assume that there is a $\kappa$-complete ultrafilter  which extending $\mathcal{D}_{\emptyset}(Q)$.

Then either there is an inner model for $\exists \lambda,\, o(\lambda) = \lambda^{++}$, or $o^{\mathcal{K}}(\kappa) > \kappa^+$.
\end{theorem}

We split the proof into three parts. First, we will derive some unconditional claims that follow from the existence of such an ultrafilter. Then, we will focus in the case that there is no inner model with a measurable $\lambda$ of Mitchell order $\lambda^{++}$, and discuss the structure of the indiscernibles that follows from the hypothesis of the theorem. Finally, we will combine those two paths and obtain a robust way to extract some of the indiscernibles, from which we are going to get strength.  

\subsection{Combinatorial consequences}
In order to prove Theorem \ref{thm:lowerbound}, we will start with a sequence of lemmas, establishing the existence of a certain elementary embedding with some useful properties.

Let $F$ be a $\kappa$-complete ultrafilter which extends $\mathcal{D}_\emptyset(Q)$.
Consider the corresponding elementary embedding $j_{F}\colon V\rightarrow \Ult(V,F)\simeq M_{F}$. Let $a=[id]_{F}$.
% and let $U_a=\{X\subseteq Q \mid a\in j_{F}(X)\}$. It follows that $U_a=F$ and
Then
\[a\in\bigcap \{j_{F}(D)\mid D\subseteq Q\text{ is dense open }\}\]
Let $a^*$ be a closed set of ordinals with minimal value of $\max(a^*)$, such that there is an ultrafilter $U=\{X\subseteq Q\mid a^*\in j_U(X)\}$ extending the filter of dense open subsets of $Q$. Equivalently, $a^*\in j_U(D)$, for every $D\subseteq Q$ dense open.

Fix such $a^*$ and let
$U$ be a witnessing ultrafilter.  So $[id]_U=a^*$, by \cite[Lemma 1.6]{Hamkins1997}, \cite[Proposition 2.5]{TomTreePrikry}.

%since, if  $f\colon Q\rightarrow Q$ is such that $[f]_U=a^*$, then, for every $X\subseteq Q$,
%$$X\in U\leftrightarrow [f]_U\in j_U(X)\leftrightarrow \{p\in Q\mid  f(p)\in X\}\in U\leftrightarrow f^{-1}\image X\in U$$
%Hence, $f$ witnesses the Rudin-Keisler equivalence of $U$ with $U$, then  it must be that $[f]_U=[id]_U$ (see \cite[Proposition 2.5]{TomTreePrikry}).

\begin{lemma}\label{minimality}
For every $\xi < \max (a^*)$, there is a dense open $D$ such that $a^*\cap(\xi+1)\notin j_U(D)$.
\end{lemma}
\begin{proof}
Otherwise, let $\xi < \max(a^*)$ be the least ordinal such that $a' = a^* \cap (\xi + 1)$ belongs to $j_U(D)$ for all $D \subseteq Q$ dense open. Let $$U'=\{X\subseteq Q\mid a'\in j_{U}(X)\}.$$ 
The ultrafilter $U'$ is below $U$ in the Rudin-Kiesler order. One way to illustrate that is to pick a function $g\colon Q \to \kappa$ such that $j_U(g)(a^*) = \xi + 1$ and define the function $f(p)= p \cap g(p)$. Then $j_U(f) = a'$. 

Let $k\colon M_{U'}\rightarrow M_{U}$ be the elementary embedding defined by  $k([f]_{U'})=j_U(f)(a')$. By standard arguments, $k\circ j_{U'}=j_U$.

Consider in $M_{U'}$ the element $b=[id]_{U'}$ (note that $U_b=U'$). By the properties of $k$, $k(b)= a' = a^*\cap(\xi+1)$, hence $\max(b)\leq k(\max(b))\leq \xi<\max(a^*)$. To see that $b$ contradicts the minimality of $a^*$, note that for every dense open $D$, $k(b) = a'\in k(j_{U'}(D))$ and by elementarity of $k$, $b\in j_{U'}(D)$.
\end{proof}
\begin{lemma}\label{genericCon}
Let $\eta<\max(a^*)$ and $q\subseteq\eta$, $q\in j(Q)$. For every dense open set $D \subseteq Q$, the condition $(a^*\setminus\eta)\cup q$ is in $j(D)$.
\end{lemma}
\begin{proof}
Otherwise, let $D_1$, $\eta<\max(a^*)$ and $q'$ be such that $$a^*\setminus\eta\cup q'\notin j(D_1)$$ By minimality of $a^*$, there is $D_2\subseteq Q$ dense and open such that $a^*\cap\eta\notin j(D_2)$.

Let $D^*$ be the set of all conditions $p\in Q$ such that there is $\eta < \max(p)$, $p \cap \eta \in D_1$ and moreover for every condition $q\in Q$ with $\max(q) \leq \eta$, $q \cup (p \setminus\eta) \in D_2$.

We claim that $D^*$ is dense open.

Let us show that $D^*$ is open. Let $p_1\in D^*$ and let $p_1\leq p_2$. Take $\eta<\max(p_1)\leq\max(p_2)$ witnessing $p_1\in D^*$ then $$p_1\cap\eta=p_2\cap\eta\in D_1$$ and if $\max(q)\leq\eta$ then $$q\cup(p_1\setminus\eta)\leq q\cup (p_2\setminus\eta)\in D_2$$ since $D_2$ is open. Thus $p_2\in D^*$.

Let us show that $D^*$ is dense. Let $p\in Q$ be any condition, find $p\leq p_1\in D_1$. denote $\max(p_1)=\eta$ and note that $2^\eta < \kappa$. Let us enumerate all $q\in Q$ with $\max(q)\leq\eta$, $\langle q_i\mid i<2^\eta\rangle$ and let $D^i_2$ be the collection of all conditions $r\in\mathbb{P}$ such that $r \setminus \eta \cup q_i \in D_2$. For every $i$, $D^i_2$ is dense open. $\mathbb{P}$ is $\kappa$-distributive and thus $\bigcap_{i < 2^\eta} D^i_2$ is dense. Let us pick a condition $p_2 \geq p_1$ in this intersection. Clearly, $p_2\in D^*$.

Let us claim that $a^*\notin j(D^*)$, and conclude the proof. For any $\xi<\max(a^*)$, if $\xi\leq \eta$ then $a^*\cap\eta\notin j(D_1)$ and if $\xi>\eta$ then let $p=a^*\cap (\eta,\xi)$ then $q'\cup p\notin j(D_2)$ since $q'\cup a^*\geq q'\cup p$. thus $a^*\notin j(D^*)$ contradiction the choice of $a^*$.
\end{proof}

 We conclude that for every $\eta$, $a^* \setminus \eta \in j(D)$ for all $D\in V$, dense open.
 In particular, we may assume that $\min(a^*)>\kappa$. Although, $a^*\setminus\kappa$ does not necessarily generates $U$, we take $b=[id]_{U_{a^*\setminus\kappa}}$, where $$U_{a^*\setminus\kappa}=\{X\subseteq Q\mid a^*\setminus\kappa\in j_U(X)\}$$ and $b$ will be as wanted, since $\max(b)=\max(a^*)$ but also $\min(b)\geq\kappa$. To see this, assume otherwise that $b\cap\kappa\neq\emptyset$. Let $k:M_{U_b}\rightarrow M_U$, then $\crit(k)\geq\kappa$ and $k(b)=a^*\setminus\kappa$ so $b\cap\kappa=k(b\cap\kappa)\subseteq a^*\setminus\kappa$, contradiction.

Continuing, we would like to derive some more information about the size of $\max (a^*)$.
\begin{lemma}\label{lem:bounding}
For any $f\colon\kappa\to \kappa$ and any $\tau<\max(a^*)$, $j_U(f)(\tau)< \max(a^*)$.
\end{lemma}
\begin{proof}
Assume otherwise, then there is $f,\tau$, witnessing the negation.
By lemma \ref{minimality}, there is a dense open set $D$ such that $a^*\cap (\tau+1)\notin j_U(D)$. 
Consider the set $D^*$ of all conditions $p\in D$ such that for every $\xi<\max(p)$, if $p\cap (\xi+1)\notin D$ then  $f(\xi)<\max(p)$. Then $D^*$ is dense since for every $p_0$, we take $q\in D$ above $p_0$, the set $\{f(\xi)\mid \xi<\max(q)\}$ is bounded by some $\max(q)\leq\alpha<\kappa$, then $$p_0\leq q\leq q^*:= q\cup\{\alpha\}\in D^*$$ since if $\xi<\max(q^*)$ and $q^*\cap(\xi+1)\notin D$, then $\xi<\max(q)$ as $q^*\cap (\max(q)+1)=q\in D$ and $D$ is open. so $f(\xi)<\alpha=\max(q^*)$. Also $D$ is open since is $p\in D^*$ and $p\leq p_1$, then $p_1\in D$ (since $D$ is open and $p\in D$), but also for every $\xi<\max(p_1)$, id $p_1\cap (\xi+1)\notin D$, then $\xi<\max(p)$. Thus $f(\xi)<\max(p)\leq \max(p_1)$. 

It follows that $a^*\in j_U(D^*)$, but this is a contradiction since $\tau<\max(a^*)$, $$a^*\cap (\tau+1)\notin j_U(D)\text{ and } j_U(f)(\tau)\geq \max(a^*)$$
\end{proof}

\begin{lemma}\label{lem-cof}
$\kappa^{+} \leq \cf^V \max(a^*) \leq 2^\kappa$.
\end{lemma}
\begin{remark}
Note that $2^\kappa>\kappa^+$ already implies, by Mitchell \cite{Mit}, that $o(\kappa)\geq \kappa^{++}$, since $\kappa$ is a measurable. Thus, assuming our anti-large cardinal hypothesis, we get $\cf \max(a^*) = \kappa^{+}$.
\end{remark}
\begin{proof}
First let us show that $\cf^V \max(a^*) \geq \kappa^+$.

Otherwise, let $\langle\zeta_\delta \mid \delta < \delta^*\rangle$ be cofinal at $\max(a^*)$, $\delta^* \leq \kappa$. For every $\delta < \delta^*$, there is a dense open set $D_\delta \in V$ such that $a^* \cap (\zeta_\delta + 1)\notin D_\delta$.

Let $D_*$ be the set of all condition $p\in D$ such that there is $\xi < \max p$ such that $p \cap \xi \in \bigcap_{\delta < \min p} D_\delta$. Clearly, $D_*$ is dense open. Let us show that $a^* \notin j_U(D_*)$. Indeed, we assume that $\min a^* > \kappa$ and therefore if $a^*\in j_U(D_*)$ then there is some $\xi < \kappa^*$ such that $a^* \cap \xi \in j_U(D_\delta)$ for all $\delta < \delta^* \leq \kappa$, which contradicts our assumption.

Let us show now that $\cf^{V}(\max(a^*)) \leq 2^\kappa$. Indeed, let us fix some elementary submodel $H$ of sufficiently large $H(\theta)$ of cardinality $2^\kappa$ that contains $a^*$ and for every $D\subseteq Q$, $j(D)\in H$. It follows that for every $D\subseteq Q$ dense open in $V$, the minimal ordinal $\rho < \max(a^*)$ such that $a^* \cap \rho \in j(D)$ belongs to $H$. In particular, $\sup (\max(a^*) \cap H) = \max(a^*)$, by the minimality of $\max(a^*)$. Since $|H\cap \max(a^*)| \leq 2^\kappa$, we conclude that $\cf (\max(a^*)) \leq 2^\kappa$.
\end{proof}

Next, we would like to get a parallel of Claim \ref{ExtensionToProj}. Since we only assume the existence of an ultrafilter extending $\mathcal{D}_{\emptyset}(Q)$, we have to be a bit more careful. We could use the homogeneity of $Q$ and derive an extension of $\mathcal{D}_p(Q)$ for all $p$, but we would like to get a relatively concrete representation of the generic, which would be useful during the proof. 

 Since $|Q|=\kappa$, there is a bijection $f:\kappa\rightarrow Q$. Denote $\delta_{a^*}=j(f^{-1})(a^*)$ and define $$\mathcal{W}=\{X\subseteq \kappa\mid \delta_{a^*}\in j(X)\}$$
then $W$ is a $\kappa$-complete ultrafilter on $\kappa$, $U\equiv^{\mathrm{RK}}\mathcal{W}$. $M_{\mathcal{W}}=M_U$ and $[g]_U\mapsto[g\circ f]_{\mathcal{W}}$ is the unique isomorphism between the two ultrapowers.
 \begin{lemma}\label{genericClub}
 Let $\langle \kappa_n\mid n<\omega\rangle$ be a  generic Prikry sequence for $\mathcal{W}$. Then $\underset{n<\omega}{\bigcup}a_n$ is a generic club for $Q$ where $a_n=f(\kappa_n)$. Moreover, there is $N<\omega$ such that for every $N\leq n<\omega$, $\max(a_n)<\min(a_{n+1})$.
  \end{lemma}
 \begin{proof}
Let $\pi\colon Q\rightarrow\kappa$ be such that $\kappa=[\pi]_{U}=[\pi\circ f]_{\mathcal{W}}$ be the projection to normal. In $V$, define the set \[A=\{\alpha<\kappa\mid \forall\beta<\pi(f(\alpha)),\, \max(f(\beta))<\min(f(\alpha))\},\] then $A\in \mathcal{W}$. To see this note that
\[M_{U}\models\forall\beta<\kappa=j(\pi)(a^*),\, \max(j(f)(\beta))<\min(a^*)\] since $j(f)(\beta)=f(\beta)<\kappa\leq\min(a^*)$. Now $a^*=[id]_{U}$, to see this, note that  $$j(\pi)(a^*)=j(\pi\circ f)(\delta_{a^*}), \ \min(a^*)=\min(j(f)(\delta_{a^*}))$$ thus  $\delta_{a^*}\in j(A)$ and $A\in \mathcal{W}$.
Let $\langle\kappa_n\mid n<\omega\rangle$ be a Prikry sequence for $W$. Then there is $N$ such that for every $N\leq n<\omega$, $\kappa_{n}<\pi(\kappa_{n+1})$ and $\kappa_n\in A$. By the definition of $A$ it follows that $\max(a_n)<\min(a_{n+1})$. Denote by $p_n=a_0\cup \cdots \cup a_n\in\mathbb{P}$, then for every $n\geq N$, $p_n\leq p_{n+1}$. We claim that $C_G=\underset{n<\omega}{\cup}a_n$ is a generic club though the singulars of $V$. To see this, let $D\subseteq\mathbb{P}$ be dense open, then by claim \ref{genericCon}, $a^*\in j(D)$ and for every $\xi<\max(a^*)$, $q\subseteq\xi$, $a^*\setminus\xi\cup q\in j(D)$, this property reflects on a set in $\mathcal{W}$ i.e. $$B=\{\alpha<\kappa\mid \forall \xi<\max(f(\alpha))\forall q\subseteq\xi , \ f(\alpha)\setminus\xi\cup q\in D\}\in \mathcal{W}$$ and therefore there is $N\leq M<\omega$ such that for every $n\geq M$, $\kappa_n\in B$ and so $p_n\in D$.
\end{proof} 

We denote $C(Q)=\bigcup_{n<\omega}a_n$ to be the $V$-generic club for $Q$.

The idea is that properties of $a^*$ reflect in some sense to the generic club $C(Q)$. This will be useful later, when we encounter some more delicate properties of $a^*$, using Mitchell's analysis of indiscernibles. 
%\begin{lemma}\label{lem0}

%$\otp(a^*)\geq \kappa$.

%\end{lemma}
%\begin{proof}
%Suppose otherwise. Then $\otp(a^*)=\rho$, for some $\rho<\kappa$.
%But this implies that for all but finitely many $n<\omega$, $\otp(a_n)=\rho$, which is impossible since $\bigcup_{n<\omega}a_n$ is a generic (over $V$) club for $Q$ and $\kappa$ remains regular in $V^Q$.
%\end{proof}

%\begin{lemma}\label{lem1}
%There is no $f:\kappa\to \kappa$ such that $j_U(f)(\kappa)\geq \max(a^*)$.

%\end{lemma}
%\begin{proof} Suppose otherwise.
%Consider a club
%$$C_f=\{\nu<\kappa \mid \forall \nu'<\nu (f(\nu')<\nu)\}.$$
%Then, in $M_U$, $(\kappa, j_U(f)(\kappa)) \cap j_U(C_f)=\emptyset$. So, $a^*  \cap j_U(C_f)=\emptyset$.
%But then  $C_f \cap C(Q)$ is bounded in $\kappa$.
%Which is impossible. Contradiction.
%\end{proof}

\subsection{Mitchell's indiscernibles}
Recall that $\mathcal{K}$ is the Mitchell's core model, under the anti-large cardinal hypothesis, $\neg \exists \lambda, o(\lambda) = \lambda^{++}$. 

For the convenience of the reader, we include here the statements of the basic definitions and results which we are going to use in the course of the proof, which we cite from \cite{MitchellHandbookCoveringLemma}. 

%To do this we will need to use Mitchell's analyses of indiscernables appropriate for the situation $o(\kappa)<\kappa^{++}$. For the convenience of the reader we have grouped here the relevant definitions and the formulation of Mitchell's covering lemma.

\begin{definition}\label{def:mitchell-ind}
\begin{enumerate}
    \item Let $U$ be a measure, then $\crit(U)$ is the measurable $\kappa$ such that $U$ is a measure over $\kappa$.
    \item Let $\vec{U}$ be a sequence of measures and let $\gamma'<\gamma$ in $\dom(\vec{U})$, denote by $coh_{\gamma',\gamma}=f$ for the least function in the well ordering of $\calK=L[\vec{U}]$ such that $\gamma'=[f]_{\vec{U}_\gamma}\in \Ult(\calK,\vec{U}_\gamma)$.
    \item A system of indiscernibles for $\calK$ is a sequence $\mathcal{C}$ such that:
    \begin{enumerate}
        \item $\dom(\mathcal{C})\subseteq\dom(\vec{U})$ and $\forall \gamma\in \dom(\mathcal{C}),\,\mathcal{C}_\gamma\subseteq \crit(\vec{U}_\gamma)$.
        \item\label{def:mitchell-ind-3-b} For every $f\in \calK$, there is a finite sets $a\subseteq On$ such that for every $\gamma\in \dom(\vec{U})$:

        $$
        \begin{matrix}
        \forall\nu\in \mathcal{C}_\gamma\setminus \sup(a\cap \crit(\vec{U}_\gamma)). \forall X\in f\image ( \nu\times\{\crit(\vec{U}_\gamma)\}) \\ \nu \in X\leftrightarrow X\cap \crit(\vec{U}_\gamma)\in \vec{U}_\gamma
        \end{matrix}$$
    \end{enumerate}
\item A sequence $\mathcal{C}$ of indiscernibles for $\calK$ is said to be $h-coherent$ if $h\in \calK$ is a function and:
\begin{enumerate}
    \item $\forall\nu\in \cup_{\gamma\in \dom(\mathcal{C})}\mathcal{C}_\gamma$, there is a unique $\xi\in h''\nu$ such that $\nu\in \mathcal{C}_\xi$.
    \item If $\nu\in \mathcal{C}_\gamma\cap \mathcal{C}_{\gamma'}$ where $\gamma\neq\gamma'$ and $\gamma\in h''\nu$, then $\crit(\vec{U}_{\gamma'})\in \mathcal{C}_{\gamma''}$ for some $\gamma''<\gamma$ with $\crit(\vec{U}_{\gamma''})=\crit(\vec{U}_\gamma)$.
    \item\label{def:mitchell-ind-4-c}  If $\nu\in \mathcal{C}_\gamma$, $\gamma_\nu=coh{\gamma',\gamma}(\nu)$ for $\gamma'<\gamma$, and $\gamma'\in h''\nu$, then $\mathcal{C}_{\gamma_\nu}=\mathcal{C}_{\gamma'}\cap(\nu\setminus\nu')$ where $\nu'$ is the least such that $\gamma\in h''\nu'$
\end{enumerate}
\item Let $x$ be any set and $h$ a function. Then set $h''(x;\mathcal{C})$ is the smallest set $X$ such that $x\subseteq X$ and $X=h''[X\cup(\bigcup_{\gamma\in X}\mathcal{C}_\gamma)]$.
\item Suppose that $\mathcal{C}$ is a $g$-coherent system of indiscernibles. Define:
\begin{enumerate}
    \item $S^{\mathcal{C}}(\gamma,\xi)=\min(\mathcal{C}_\gamma\setminus\xi+1)$.
    \item $S^{\mathcal{C}}_*(\gamma,\xi)=\min(\bigcup_{\gamma'\geq\gamma}\mathcal{C}_\gamma\setminus\xi+1)$.
    \item If $X$ is any set, and $\gamma\in \dom(\mathcal{C})\cap X$. An accumulation point of $\mathcal{C}_\gamma$ in $X$ is an ordinal $\nu\in X$ such that for every $\gamma'\in X\cap \gamma\cap g''\nu$, the $$\bigcup\{\mathcal{C}_{\gamma''}\mid \gamma''\geq \gamma', \crit{\vec{U}_{\gamma''}}=\crit{\vec{U}_\gamma}\}$$
    is unbounded in $\nu$. Let $a^{\mathcal{C},X}(\gamma,\xi)$ is the least accumulation point of $\mathcal{C}_\gamma$ in $X$ above $\xi$.
\end{enumerate}
\end{enumerate}
\end{definition}
\begin{theorem}[Mitchell's Covering Lemma]\label{the covering lemma}
Assume there is no inner model with $\exists \lambda. o(\lambda)=\lambda^{++}$. Let $\kappa,\kappa'$ be a $\calK$-cardinal such that $\kappa'\geq \max\{\kappa, o(\kappa)\}$. Also let $X$ be a set such that $\kappa\not\subseteq X=Y\cap\calK_{\kappa'}$ where $Y\prec_1 H(\kappa^{'+})$. Then there is $\rho<\kappa$, $h\in\calK$ and $\mathcal{C}$ such that:
\begin{enumerate}
    \item $\mathcal{C}$ is an $h$-coherent system of indiscernibles for $\calK$.
    \item $\dom(C)\subseteq X$ and $\cup_\gamma \mathcal{C}_\gamma\subseteq X$.
    \item $X=h''(X\cap\rho;\mathcal{C})$ and hence $X\subseteq h''(\rho;\mathcal{C})$.
    \item For every $\nu\in X\cap \kappa$, either $\nu\in h''[X\cap\nu]$, or $\nu\in\mathcal{C}_\gamma$ for some
    $\gamma$ in which case there is $\xi\in X\cap\nu$ such that either
    $\nu=S^{\mathcal{C}}(\gamma,\xi)=S^{\mathcal{C}}_*(\gamma,\xi)$ or there is
    $\gamma<\gamma'\in h''[X\cap\nu]$ such that $\nu=a^{\mathcal{C},X}(\gamma',\xi)$.
    \item If $X'$ is another set satisfying is another set satisfying the assumption of the theorem then there is a finite set $a\subseteq On$ such that for every $\xi,\gamma\in X\cap X'$ such that $a\cap \crit(\vec{U}_\gamma)\subseteq \xi$ and $\xi>\max\{\rho_X,\rho_{X'}\}$ then:
    $$S^{\mathcal{C}}(\gamma,\xi)=S^{\mathcal{C}'}(\gamma,\xi)$$
    $$S^{\mathcal{C}}_*(\gamma,\xi)=S^{\mathcal{C}'}_*(\gamma,\xi)$$
    $$a^{\mathcal{C},X}(\gamma,\xi)=a^{\mathcal{C}',X'}(\gamma,\xi)$$
    whenever either is defined.
\end{enumerate}
\end{theorem}
A \textit{Covering model} is a set $X$ satisfying the assumptions of the theorem \ref{the covering lemma}. In the discussion ahead, we will not distinguish between a model and its set of ordinals. Thus, we will freely take elementary substructures in some model of $\ZFC$, that do not contain all ordinals below some $\zeta$ and call them covering models.

The elementary embedding $j_U\restriction\mathcal{K}\colon \mathcal{K} \to \mathcal{K}^{M_U}$ is an iterated ultrapower of $\mathcal{K}$ by its measures.
Let us denote the iteration by $\l j_{\alpha,\beta}\mid \alpha\leq \beta\leq l^*\r$ where $j_{\alpha,\beta}:\mathcal{K}_{\alpha}\rightarrow \mathcal{K}_{\beta}$. We can assume that the iteration is normal i.e. $\l \lambda_i\mid i< l^*\r$ is increasing where $\lambda_i=\crit(j_{i,i+1})$. Hence $\lambda_0=\kappa$. Let $\langle \kappa_\alpha \mid \alpha\leq \alpha^*\rangle$ be the strictly increasing list of images of $\kappa$ under this iteration.
In particular, $\kappa_0=\kappa$ and $\kappa_{\alpha^*}=j_U(\kappa)$, and $\alpha^* \leq l^*$. 

\begin{lemma}\label{crit}
Let $\eta<\alpha^*$ and let $\eta_0$ be the least ordinal such that $\kappa_{\eta}\leq\lambda_{\eta_0}$. Also let $\xi_\eta<l^*$ be such that $j_{\xi_\eta}(\kappa)=\kappa_{\eta}$  Then $j_{\xi_\eta,\eta_0}(\kappa_{\eta})=\kappa_{\eta}$
\end{lemma}
\begin{proof}
By elementarity, $N_0:=\calK_{\xi_\alpha}\models \kappa_{\eta}$ is measurable. Let us define an internal iteration of the measures of $N_0$, $i:N_0\rightarrow N_0^*$, $\l i_{\alpha,\beta}\mid \alpha\leq\beta\leq \theta^*\r$ defined as follows:

At limit steps we simply take a direct limit.
At successor step $\gamma+1$, assume that $i_{0,\gamma}\colon N_0\rightarrow N_\gamma$ is defined and $U_\beta$ is a measure of $\zeta_\beta$ for $\beta<\gamma$ are the measures applied at stage $\beta$. 

Let \[\bar{\zeta}_\gamma=\sup_{\beta<\gamma}(\zeta_\beta+1).\] 
We split into cases:
\begin{itemize}
    \item If $\cf^{N_0}(\gamma)>\kappa$ or $\gamma$ is successor ordinal, consider the first measurable $\zeta_\gamma\geq\bar{\zeta}_\gamma$ in $N_\gamma$ and apply $U(\zeta_\gamma,0)$. 
    \item If $\cf^{N_0}(\gamma) \leq \kappa$ and $\gamma$ is a limit ordinal, we take the least $N_\gamma$-measurable $\zeta=\zeta_\gamma$, such that for some $\rho$, the set $\{\beta<\gamma\mid i_{\beta,\gamma}(U_\beta)=U(\zeta,\rho)\}$ is bounded in $\gamma$, assuming that there is one. If there is no such $\zeta$, take $\gamma=\theta^*$ and halt.
\end{itemize}

Let us claim that the elementary embedding $j_{\xi_\eta, l^*}\colon N_0\rightarrow\calK^{M_U}$ can be completed to $N^*$. Indeed, in the comparison process between the models $N^*$ and $\calK^{M_U}$, the model $N^*$
will not move since measurable cardinal in $N^*$ are critical points of steps of the iteration of cofinality at most $\kappa$ and $M_U$ is closed under $\kappa$-sequences. 

Hence there is an iteration
$\l\sigma_{\alpha,\beta}\mid \alpha\leq\beta\leq \rho^*\r$ such that $\sigma_{\rho^*}\circ j_{\xi_\eta, l^*}=i$. We are only interested in the part of the iteration which have critical points below $\kappa_\eta$, and the iteration $\sigma\circ j_{\xi_\eta, l^*}$
is equivalent to a normal one. Let $\beta_0$ be the least such that  $\crit(i_{\beta_0,\beta_0+1})\geq\kappa_\eta$, then there is $\gamma_0$ such that $\sigma_{\gamma_0}\circ j_{\eta_0,\xi_\eta}=i_{\beta_0}$. %We define an iteration $\sigma:\calK_{\alpha_0}\rightarrow N^*$. Inductively assume that after $\beta$ many steps we did not move $N^*$ and $\sigma_\beta:\calK_{\alpha_0}\rightarrow M_\beta$ is defined. We only need to prove that every measure in $N^*$ is in $M_\beta$. Let $U$ be a measure in $N^*$ on $\lambda$. Then $\lambda$ is a critical point $\lambda=crit(i_{\rho,\rho+1})$. It is impossible that $\rho$ is successor, nor that $\rho$ is of $N_0$-cofinality grater than $\kappa$ since otherwise, $i_{\rho,\rho+1}$ destroys the measurably of $\lambda$ and in turn $\lambda$ is not measurable in $N^*$. Otherwise, the $N_0$-cofinality of $\rho$ is at most $\kappa$,  Let $\beta_0$ be the least such that $\forall\beta<\beta_0$, $\zeta_\beta<\kappa_\alpha$, and let us argue that $i_{\beta_0}(\kappa_\alpha)=\kappa_\alpha$. Clearly $i_{\beta_0}(\kappa_\alpha)\geq \kappa_\alpha$,  
Since $i_{\beta_0}$ is an internal iteration of $N_0$ with critical points below $\kappa_{\eta}$ which is measurable in $N_0$,  $i(\kappa_{\eta})=\kappa_{\eta}$. 

Hence $$\kappa_{\eta}\leq j_{\eta_0,\xi_\eta}(\kappa_{\eta})\leq \sigma_{\gamma_0}(j_{\eta_0,\xi_{\eta}}(\kappa_{\eta}))=\kappa_{\eta}.$$ 
We conclude that $j_{\xi_\eta, \eta_0}(\kappa_{\eta})=\kappa_{\eta}$.
\end{proof}
%Also , for limit $\alpha\leq\theta^*$ we take direct limit. At successor step $\alpha$, we have defined the model $N_\beta$ for $\beta\leq\alpha$ and the measures used are denoted by $U_\beta$. .

%We can complete the iteration $j_U\restriction\calK$ to the full iteration.
%From $\calK^{M_U}$, we define an external iteration %$\l\sigma_{\alpha,\beta}\mid \alpha\leq\beta\leq \zeta^*\r$, which commutes with the iteration $j_U\restriction\calK$, 
%At limit we take direct limit, and at successor step we take the minimal measurable $\mu$ such that $\{\beta<\alpha\mid \sigme_{\beta,\alpha}(U_\beta)=U(\mu,\rho)\}$ is bounded in $\alpha$.
%So $\sigma\circ j_U\restriction\calK=i$, and we can  ????
\begin{corollary}\label{corcrit} $\{\kappa_\alpha\mid \alpha\leq \alpha^*\}\subseteq \{\lambda_i \mid i\leq l^*\}\cup\{j_U(\kappa)\}$.
\end{corollary}
\begin{proof}
 Assume that  $\kappa_\alpha\notin\{\lambda_i \mid i\leq l^*\}$, let us show that $\kappa_\alpha=j_U(\kappa)$. Let $\xi_\alpha$ to be the least such that $j_{\xi_\alpha}(\kappa)=\kappa_\alpha$. Consider $\alpha_0$ to be the minimal such that $\kappa_\alpha\leq\lambda_{\alpha_0}$. If $\alpha_0=l^*$, then we are done. Otherwise we actually get the conclusion by using the assumption that $\kappa_\alpha<\lambda_{\alpha_0}=\crit(j_{\alpha_0, l^*})$. Clearly, $\xi_\alpha\leq \alpha_0$, otherwise, since $\crit(j_{\alpha_0,\xi_\alpha})=\lambda_{\alpha_0}>\kappa_\alpha$ (again, this is clear in case $\alpha_0=l^*)$, $j_{\xi_\alpha,\alpha_0}(\kappa_\alpha)=\kappa_\alpha=j_{\xi_\alpha,\alpha_0}(j_{\alpha_0}(\kappa))$, hence $\kappa_\alpha=j_{\alpha_0}(\kappa)$, contradiction the minimality of $\xi_\alpha$.
 By lemma \ref{crit}, $j_{\alpha_0,\xi_\alpha}(\kappa_\alpha)=\kappa_\alpha$, hence
$$j_U(\kappa)=j_{l^*,\alpha_0}(j_{\alpha_0,\xi_\alpha}(j_{\xi_\alpha}(\kappa)))=j_{\alpha_0,l^*}(j_{\xi_\alpha,\alpha_0}(\kappa_\alpha))=j_{\alpha_0,l^*}(\kappa_\alpha)=\kappa_\alpha$$
\end{proof}

 \begin{claim}\label{BoundFun}
If $\kappa_\alpha\leq\delta<\kappa_{\alpha+1}$ then there is $h\in({}^\kappa\kappa)^{\mathcal{K}}$ such that $\delta\leq j_U(h)(\kappa_\alpha)<\kappa_{\alpha+1}$.
\end{claim}
\begin{proof}
 Assume $\kappa_\alpha\leq\delta<\kappa_{\alpha+1}$ decompose the iteration $$j_U\restriction\calK=j_{\xi_{\alpha+1},l^*}\circ j_{\xi_{\alpha}+1,\xi_{\alpha+1}}\circ j_{\xi_\alpha,\xi_\alpha+1}\circ j_{\xi_\alpha}$$ where $$j_{\xi_\alpha}:\calK\rightarrow \calK_{\xi_\alpha}, \ \crit(j_{\xi_\alpha})=\kappa$$  $$j_{\xi_\alpha,\xi_\alpha+1}:\calK_{\xi_\alpha}\rightarrow \calK_{\xi_\alpha+1}, \ \crit(j_{\xi_\alpha,\xi_\alpha+1})=\kappa_\alpha,
 \text{ and }j_{\xi_\alpha,\xi_\alpha+1}(\kappa_\alpha)=\kappa_{\alpha+1}$$
 $$j_{\xi_\alpha+1,\xi_{\alpha+1}}:\calK_{\xi_\alpha+1}\rightarrow\calK_{\xi_{\alpha+1}}, \crit(j_{\xi_\alpha+1,\xi_{\alpha+1}})=\lambda_{\xi_\alpha+1}$$
 $$j_{\xi_{\alpha+1},l^*}:\calK_{\xi_{\alpha+1}}\rightarrow \calK^{M_U}, \ \crit(j_{\xi_{\alpha+1},l^*})=\kappa_{\alpha+1}$$
 First consider only the iteration $j_{\xi_\alpha+1}$, there is $f\in({}^\kappa\kappa)^\calK$ such that $$j_{\xi_\alpha+1}(f)(\lambda_{i_1},\dots,\lambda_{i_n})=\delta$$ where $\lambda_{i_1},\dots,\lambda_{i_n}\leq\kappa_\alpha$. Now let us define $h:\kappa\rightarrow\kappa$ by $$h(\alpha)=\sup(f(\vec{\xi})\mid \vec{\xi}\in[\alpha+1]^n)$$
 $h\in\calK$ as it is definable. It follows that $\delta\leq j_{\xi_\alpha+1}(h)(\kappa_\alpha)<\kappa_{\alpha+1}$. 
 Further iteration might move $j_{\xi_\alpha+1}(h)(\kappa_\alpha)$, but not past $\kappa_{\alpha+1}$. Indeed, by lemma \ref{crit}, $$\kappa_{\alpha+1}\leq j_{\xi_{\alpha}+1,\xi_{\alpha+1}}(\kappa_{\alpha+1})\leq j_{\xi_{\alpha+1}}(\kappa_{\alpha+1})=\kappa_{\alpha+1}$$
 Hence $\kappa_{\alpha+1}= j_{\xi_{\alpha}+1,\xi_{\alpha+1}}(\kappa_{\alpha+1})$. It follows that $$j_U(h)(\kappa_\alpha)=j_{\xi_{\alpha+1},l^*}(j_{\xi_\alpha+1}(h)(\kappa_\alpha))=j_{\xi_{\alpha+1}}(h)(\kappa_\alpha)=$$
 $$=j_{\xi_\alpha+1,\xi_{\alpha+1}}(j_{\xi_{\alpha}+1}(h)(\kappa_\alpha))<j_{\xi_\alpha+1,\xi_{\alpha+1}}(\kappa_{\alpha+1})=\kappa_{\alpha+1}$$
 \end{proof}
 There is a close connection between the critical points of the iteration $j_U$ and indiscernibles of covering models  from Mitchell's covering lemma. 
 
\begin{lemma}\label{lemma: finite error}
Let $N=h^{N}\image(\rho;\mathbb{C}^N)$ be a covering model where $\mathbb{C}^N$ is a $h^N$-coherent system of indiscernibles for $\calK_{M_U}$ where $h^N \in \calK^{M_U}$ is a Skolem function. Suppose that $\kappa_{\gamma_0}\in N$ for some $\gamma_0<\alpha^*$. 

Then
 for all but finitely many $c \in \cup\{ \mathbb{C}^N_\gamma\mid \crit(\mathbb{C}^N_\gamma)=\kappa_{\gamma_0}\}$, $c\in \{\kappa_\alpha\mid \alpha\leq\gamma_0\}$.
\end{lemma}
\begin{proof}
 Suppose otherwise. Let $\l \delta_n \mid  n < \omega\r$ be an increasing sequence in $$\cup\{ \mathbb{C}^N_\gamma\mid \crit(\mathbb{C}^N_\gamma)=\kappa_{\gamma_0}\}\setminus  \{\kappa_\alpha\mid \alpha\leq\alpha^{*}\}$$
 Set $$\alpha_n = \max(\{\alpha\leq \alpha^{*} \mid \kappa_\alpha<\delta_n\}).$$
 By Claim \ref{BoundFun} there is $f_n : \kappa\to \kappa$ in $\calK$ increasing such that
  $$\delta_n < j_U(f_n)(\kappa_{\alpha_n}) <\kappa_{\alpha_{n}+1}.$$ Consider $\{f_n \mid n < \omega\}.$
   While this set might not be a member of $\calK$, we are above to bound it. Let $\l t_\xi \mid \xi < \kappa^+\r$
    be the canonical enumeration in $\calK$ of $(\kappa^\kappa)^\calK$. For every $n < \omega$, let $\xi_n$ be the unique ordinal such that $f_n = t_{\xi_n}$. Both $\kappa$ and $\kappa^+ = (\kappa^+)^{\calK}$ are regular in $V$ (here we are using the covering theorem, and the measurability of $\kappa$ in $V$). So, there is $a \subseteq\kappa^+, a \in \calK,|a| < \kappa$ which covers
    $\{\xi_n \mid n < \omega\}$. To find such a set, let $\xi=\sup_n \xi_n < \kappa^{+}$.
Let $p\colon \kappa \to \xi$ be a bijection in $K$. Since $\cf^V \kappa = \kappa > \omega$, $\sup p^{-1}(\xi_n) = \beta' < \kappa$.
Then take $a= p\image \beta' \in K$. 

Define a function $f\colon\kappa\to \kappa$ in $\calK$ as follows:
     $$f(\nu) = \sup\{t_\xi(\nu) \mid \xi\in  a\},$$ for every $\nu<\kappa$. Then, for every $n < \omega,\nu<\kappa,
$ $$\kappa > f(\nu) > f_{\xi_n}(\nu).$$

Now, in the ultrapower, for every $n < \omega,$
$$\kappa_{\alpha_n+1}> j_U(f)(\kappa_{\alpha_n}) > \delta_n.$$

Let $\delta^*=\sup_{n<\omega}\delta_n\leq\kappa_{\gamma_0}$. If $\delta^*=\kappa_{\gamma_0}$, then the function $j_U(f)\image\nu$ in is $\calK^{M_U}$. Note that $\delta_n\in  j_U(f)\image\kappa_{\alpha_n}$. For high enough $n$, this will contradicts Definition \ref{def:mitchell-ind}, \ref{def:mitchell-ind-3-b} and the indiscernibility of $\delta_n$'s.
If $\delta^*<\kappa_{\alpha+1}$, then it is also indiscernible and by definition \ref{def:mitchell-ind}, \ref{def:mitchell-ind-4-c}, the $\delta_n$'s are part of the indiscernibles for $\delta^*$. Then we again reach a contradiction to \ref{def:mitchell-ind}, \ref{def:mitchell-ind-3-b}.
\end{proof}
 
 \begin{lemma}
 For every $\alpha<\alpha^*$, $\kappa_{\alpha+1}$ is regular in $M_U$.
 \end{lemma}
 \begin{proof}
 Otherwise, it is singular in $M_U$, denote by $\lambda=\cf^{M_U}(\kappa_{\alpha+1})<\kappa_{\alpha+1}$.

Work in $M_U$, let $H\prec H(\theta^+)$ be an elementary submodel for some high enough $\theta$, closed to $\lambda$ sequences, such that $|H|<\kappa_{\alpha+1}$. Apply Mitchell's covering lemma \ref{the covering lemma}, find a covering
 model $H\cap \calK\subseteq N$ of cardinality less than $\kappa_{\alpha+1}$. It is of the form $h^N \image(\delta^N, \mathbb{C}^N)$, where $\delta^N<\kappa^*$,
$\mathbb{C}^N$ is a $h^N$-coherent system of indiscernibles for $\calK_{M_U}$ and $h^N \in \calK^{M_U}$ is a Skolem function.
We can assume also that $\lambda\subseteq H$. The indescernibles for $\kappa_{\alpha+1}$ in $N$ are unbounded in $\kappa_{\alpha+1}$. On the other hand, all but finitely many indiscernables for $\kappa_{\alpha+1}$ are among $\{\kappa_{\beta}\mid \beta\leq\alpha\}$. This is a contradiction.
 \end{proof}

Consider $\kappa^*=\sup(\max(a^*)+1 \cap \{ \kappa_\alpha \mid \alpha\leq \alpha^*\})$. Then there is $\alpha^{**}<\alpha^*$ such that $\kappa^*=\kappa_{\alpha^{**}}$.

%\begin{lemma}\label{lem:len_larger_than_kappa}
$\alpha^{**}\geq \kappa$.
\\In particular, the length of the sequence $\langle \kappa_\alpha \mid \alpha\leq \alpha^*\rangle$ is at least $\kappa$, and hence, $o^{\calK}(\kappa)\geq \kappa$. We will not pursuit that direction here, as the next lemma gives a strictly stronger result.

\begin{lemma}
$\kappa^*=\max(a^*)$.
\end{lemma}
\begin{proof}
Otherwise, $\kappa^*<\max(a^*)<\kappa_{\alpha^{**}+1}$. By claim \ref{genericCon}, for every $D$ dense open, $a^*\setminus\kappa^*\in j(D)$. Also, by minimality of $a^*$, there is a dense open set $D_0$ such that $a^*\cap(\kappa^*+1)\notin D_0$.

Let $h\colon\kappa\rightarrow\kappa$ be such that $j(h)(\kappa^*)\geq\max(a^*)$ which exists by claim \ref{BoundFun}.
Consider
\[C=\{\alpha<\kappa\mid \forall\beta<\alpha,\, h(\beta)<\alpha\},\]
the club of all closure points of $h$.

%Since $\kappa^* < \max a^*$, by Claim \ref{genericCon} there is a dense open set $D_0$ such that $$a^* \cap (\kappa^* + 1) \notin j(D_0)$$ 
Let $D$ be the dense open set of all conditions $p\in Q$ such that there are $$\eta < \eta' < \max p$$ such that $p \cap \eta \in D_0$ and $\eta' \in C$. Let us claim that $a^* \notin j(D)$, and thus obtain a contradiction. Indeed, the least $\eta$ such that $a^* \cap \eta \in D_0$ is above $\kappa^* + 1$ and the next element of $j(C)$ above $\kappa^* + 1$ is at least $\max a^*$.
\end{proof}

\begin{claim}\label{claim:bounding-using-function-from-K}
For every $\alpha < \alpha^*$, and a function $f \colon \kappa_\alpha \to \kappa_\alpha$ in $\mathcal{K}^{M_U}$, there is a function $g \in \mathcal{K}$ such that $j(g)(\zeta) \geq f(\zeta)$ for all $\zeta < \kappa_{\alpha}$, except for a bounded error.  
\end{claim}
\begin{proof}
Fix in $\calK = \calK^V$ a sequence $\l h_\tau \mid \tau<\kappa^+\r$ of functions such that for every $\tau<\tau'<\kappa^+$ the following hold in $\calK$:

\begin{enumerate}
  \item $h_\tau:\kappa\to \kappa$,
  \item $h_\tau<h_{\tau'}$ mod bounded,
  \item for every $g:\kappa\to \kappa$ there is $\rho<\kappa^+$ such that $g<h_{\rho}$ mod bounded.
\end{enumerate}

Note that $2^\kappa=\kappa^+$ in $\calK$, hence it is easy to construct such a sequence.
Apply the iteration  $j_U\restriction \calK$ to the list  $\l h_\tau \mid \tau<\kappa^+\r$. %Recall that $\kappa^*=\kappa_{\alpha^{**}}$, for some $\alpha^{**}<\alpha^*$ and 
Let us denote by $\calK'$ the iterated ultrapower of $\calK$, and $i\colon \calK \to \calK'$ the iteation, so that $i(\kappa) = \kappa_\alpha$, and the critical point of the rest of the iteration is $\geq \kappa_\alpha$.

Note that $i\image \kappa^+$ is cofinal at $i(\kappa^+)$. Moreover, $\kappa^+=(\kappa^+)^\calK$, by the anti-large cardinal assumptions made.
Hence $\l i(h_\tau) \mid  \tau<\kappa^+\r$ will be dominating family of functions from $\kappa_\alpha$ to $\kappa_\alpha$ in $\mathcal K'$. As the critical point of the rest of the iteration
is high enough, $j(h_\tau) \restriction \kappa_\alpha = i(h_\tau)$. 
%In particular, there will be $\tau^*<\kappa^+$ such that $i(h_{\tau^*})> \tilde{h}^N$ mod bounded.
\end{proof}

\subsection{Isolating the indiscernibles}

%We will extend here the arguments of the previous section in order to prove Theorem \ref{thm:lowerbound} .
Recall that $\l \kappa_\beta \mid \beta\leq\alpha^{**}\r$ is the sequence of images of $\kappa$ under the iterated ultrapower $j_U\upr \calK$. In particular, each $\kappa_{\beta+1}$ is the image of $\kappa_\beta$ under the ultrapower embedding using a measure over $\kappa_\beta$,
$\kappa_0=\kappa$ and $\kappa_{\alpha^{**}}=\kappa^*$.

%The first problematic case for the argument of the previous section is $o^{\calK^{M_U}}(\kappa^*)=\kappa^*+\kappa^+$.
%The problem here is that the covering arguments are based on using  models of size less than $\kappa^*$.
%In the situation when $o^{\calK^{M_U}}(\kappa^*)<\kappa^*$, or more precisely, if $\alpha^{**}<\kappa^*$, a single model can be used to cover everything relevant, namely $a^*\cap \{\kappa_\alpha \mid \alpha<\alpha^{**}\}$, as it was done in the previous section. The set $\{\kappa_\alpha \mid \alpha<\alpha^{**}\}$ was relatively small.
%Now, this  set    may be much bigger, and so it will be impossible to catch it in a single covering model.

%Actually this type of problem occurs already if  $o^{\calK^{M_U}}(\kappa^*)=\kappa^*$, but we know that $\cof^V(\kappa^*)>\kappa$, so it is possible to proceed
%further up to $o^{\calK^{M_U}}(\kappa^*)=\kappa^*+\kappa^+$. 
The following lemma provides a sufficient condition for the main theorem of this section:
\begin{lemma}\label{lem:isolating-the-indiscernibles}
Let $A(\eta)=\{\kappa_\gamma\mid \kappa_\gamma<\eta\}\cap \acc(a^*)$.
 
 If there function $t\in({}^{\kappa^*}\kappa^*)^{\calK^{M_U}}$ and $\gamma<\kappa^*$ such that $A(\kappa^*)\setminus\gamma=C_t\cap \acc(a^*)\setminus\gamma$, then $o^{\calK^{M_U}}(\kappa^*)\geq(\kappa^*)^+$,
\end{lemma}
\begin{proof}
Assume otherwise that $o^{\calK^{M_U}}(\kappa^*)<(\kappa^*)^+$.
Using Claim \ref{claim:bounding-using-function-from-K}, we find some $t^*\in\calK$ such that $j_U(t^*)$ dominates $t$. 
Find disjoint sets $\l X_i\mid i<o^{\calK}(\kappa)\r$ such that 
$X_i\in U(\kappa,i)$. Since $o^{\calK}(\kappa)<\kappa^+$ there is a bijection $\pi:o^{\calK}(\kappa)\rightarrow \kappa$.
Define $g:\kappa\rightarrow \kappa$ by $g(\nu)=\pi(i)$ for the unique $i$ such that $\nu\in X_i$.
Let us argue that $$(\star) \ \ A^*:=\{\nu<\kappa\mid g(\nu)<\nu\}\in\cap_{\xi<o^{\calK}(\kappa)} U(\kappa,\xi).$$
Let $\xi<o^{\calK}(\kappa)$, then
in the ultrapower $Ult(\calK,U(\kappa,\xi))$, $j_{U(\kappa,\xi)}(g)$ is defined similarly using $j_{U(\kappa,\xi)}(\pi):j_{U(\kappa,\xi)}(o^{\calK}(\kappa))\rightarrow j_{U(\kappa,\xi)}(\kappa)$ and the sequence $$j_{U(\kappa,\xi)}(\l X_i\mid i<o^{\calK}(\kappa)\r)=\l X'_i\mid i<j_{U(\kappa,\xi)}(o^{\calK}(\kappa))\r$$
Note that $\kappa\in j_{U(\kappa,\xi)}(X_\xi)= X'_{j_{U(\kappa,\xi)}(\xi)}$ hence $$j_{U(\kappa,\xi)}(g)(\kappa)=j_{U(\kappa,\xi)}(\pi)(j_{U(\kappa,\xi)}(\xi))=j_{U(\kappa,\xi)}(\pi(\xi))=\pi(\xi)<\kappa$$ which is what we needed.

By $(\star)$, we can deduce that that $\forall \alpha<\alpha^{**}$, $j_U(g)(\kappa_\alpha)<\kappa_\alpha$. In particular $$(\star\star)  \ \ \ M_U\models \ j_U(g)\text{ is regressive on } \acc(a^*)\cap j_U(C_{t^*})$$

Using our hypothesis again, $\acc(a*) \cap j_U(C_{t*})$ consists of the indiscernibles of $\kappa^*$. In particular, if $\langle \alpha_n \mid n < \omega\rangle$ is a sequence of ordinals below $\alpha^{**}$, such that $j_U(g)(\kappa_{\alpha_n})$ is fixed, then $j_U(g)(\sup \kappa_{\alpha_n})$ is strictly higher.  

Let $\l a_n\mid n<\omega\r$ be a Prikry sequence for $U$ obtained by lemma \ref{genericClub} and let $C(\mathcal{Q})=\cup_{n<\omega} a_n$ be the generic club induced for $\mathcal{Q}$.
By reflecting $(\star\star)$, we get that for every $n\geq n_0$, $g$ is regressive on $\acc(a_n)\cap C_{t^*}$. Hence in $V[C(\mathcal{Q})]$, $g$ is regressive on a final segment of $C(\mathcal{Q})\cap C_{t^*}$ which is a club in $V[C(\mathcal{Q})]$. Since $\kappa$ remains regular in the generic extension $V[C(\mathcal{Q})]$, there is a stationary subset $S \subseteq C_{t*} \cap \acc C(\mathcal{Q})$ of ordinals of countable cofinality, on which $g$ is fixed. But, in particular, there is a continuous copy of $\omega + 1$ that consists of elements of $S$---a contradiction.
\end{proof}
Note that if $o^{{\mathcal K}^{M_U}}(\kappa^*) \geq (\kappa^*)^+$, then $o(\kappa) > \kappa^{+}$.
Indeed, $\kappa^*=\max(a^*)<j_U(\kappa)$. It follows that there is $\xi<l^*$ such that $\kappa^*=\crit(j_{\xi,\xi+1})$. This means that $o^{\calK_\xi}(\kappa^*)>o^{\calK^{M_U}}(\kappa^*)\geq(\kappa^*)^+$. By elementarity, $o^{\calK}(\kappa)>\kappa^+$.  
%Now for the main result of this section:
%\begin{theorem}\label{main theorem of section}

% Let us assume that there is a $\kappa$-complete ultrafilter  which extending $\mathcal{D}_\emptyset(Q)$.

%Then either there is an inner model for $\exists \lambda, o(\lambda) = \lambda^{++}$, or $o^{\calK}(\kappa) > \kappa^+ $.
%\end{theorem}
So, in order to conclude the proof, we need to prove that the hypothesis of Lemma \ref{lem:isolating-the-indiscernibles} holds.  
\begin{lemma}
For every $\eta \leq \alpha^{**}$, such that $\kappa_\eta \in \acc(a^*)$, there are $t_\eta \in ({}^{\kappa_\eta}{\kappa_\eta})^{{\mathcal K}^{M_U}}$ and $\gamma_\eta$ such that $A(\eta) \setminus \gamma_\eta= (C_{t_\eta} \cap a^* \cap \kappa_\eta) \setminus \gamma_\eta$. 
\end{lemma}
\begin{proof}
First note that if $\eta$ is a limit ordinal and $t$ is any function from $\kappa_\eta$ to $\kappa_\eta$ in $\mathcal{K}^{M_U}$, then for every sufficiently large $\alpha < \eta$, $\kappa_\alpha \in C_t$. This is true, by the arguments of the proof of Claim \ref{BoundFun} --- each such function $t$ is obtained by plugging into a an $j$-image of a function in $\mathcal{K}$, finitely many ordinals below $\kappa_{\eta}$, and restricting it to $\kappa_\eta$.  

We prove by induction in $\eta\leq\alpha^{**}$ such that $\kappa_\eta\in \acc(a^*)$. 

%that if $o^{\calK^{M_U}}(\kappa_\eta)<(\kappa_\eta^+)^{M_U}$ then there is $t^{\eta}$ and $\mu$ such that $C_{t^{\eta}}\cap Lim(a^*)\setminus\mu=A(\eta)\setminus\mu$. 
%Then we will conclude the existence of $t^{\alpha^{**}}$ and then we apply lemma \ref{Lemma for lower bound} to conclude that $o^{\calK^{M_U}}(\kappa^*)\geq(\kappa^*)^+$. 

%Let $\eta\leq \alpha^{**}$ such that $\kappa_\eta\in \Lim(a^*)$,  and $o^{\calK^{M_U}}(\kappa_\eta)<(\kappa_\eta^+)^{M_U}$. Hence there is $\gamma_0<\eta$ such that for every $\gamma_0\leq\gamma<\eta$, $o^{\calK^{M_U}}(\kappa_\gamma)<(\kappa_\gamma)^+$. 
Assume inductively that the claim holds for all $\eta'<\eta$.  Since $a^*$ is closed, $\kappa_\eta\in a^*$, thus $\kappa_\eta$ is singular in $M_U$. Let us denote by $\lambda=\cf^{M_U}(\kappa_\eta)<\kappa_\eta$ and split into cases:
\vskip 0.3 cm

\textbf{Case 1}: Assume that  $\lambda>\omega$.
   Since $\kappa_\eta$ is measurable in $\calK^{M_U}$ and singular in $M_U$ there is a Prikry-Magidor sequence in $\l c_i\mid i<\lambda\r\in M_U$ witnessing the singularity of $\kappa_\eta$. We can cover $\{c_i\mid i<\lambda\}$ with a covering model $N$ for $\kappa_\eta$ of cardinality less than $\kappa_\eta$ such that all the $c_i$'s are indiscernibles for $\kappa^*$ is $N$. By lemma \ref{lemma: finite error}, for all but finitely many indiscernibles for $\kappa_\eta$, $c_i\in \{\kappa_\gamma\mid \gamma<\eta\}$. By removing a bounded piece if necessary, we can assume that $c_i=\kappa_{\eta_i}$, for some ordinals $\eta_i$.% and $o^{\calK^{M_U}}(\kappa_{\eta_i})<(\kappa_{\eta_i}^+)^{M_U}$. 
   Since both $a^*\cap \kappa_\eta$ and $\{\kappa_{\eta_i}\mid i<\lambda\}$ are clubs in $\kappa_\eta$ inside $M_U$, and the cofinality of $\kappa_\eta$ is $\lambda>\omega$, we my also assume that each $\kappa_{\eta_i}$ is a limit point of $a^*$. Apply the inductive hypothesis to each of the points $\kappa_{\eta_i}$ and obtain a function $t^{\eta_i}\colon\kappa_{\eta_i}\rightarrow\kappa_{\eta_i}$ in $\calK^{M_U}$ such that for some $\nu_i<\kappa_{\eta_i}$, $$C_{t^{\eta_i}}\cap \Lim(a^*)\setminus \nu_i=A(\eta_i)\setminus \nu_i.$$ 
  By Mitchell's covering lemma, \ref{the covering lemma}, $N=h^N \image (\delta^N, \mathbb{C}^N)$, where $\delta^N<\kappa_\eta$,
  $\mathbb{C}^N$ is a sequence of indiscernibles and $h^N \in \calK^{M_U}$ is a Skolem function.
  
  In order to find a single function that works for $\kappa_\eta$ we will prove that we can choose these functions $t^{\eta_i}$ so that they are definable in the covering model $N$.%, then we use $\tilde{h}^N$ from the previous section.
  
  First let us argue that $A(\eta_i)$ is definable in $M_U$ from the parameters $\kappa_{\eta_i}$ and $a^*$, up to an initial segment, and this definition is uniform. 

\begin{lemma}\label{lem5-2-0-1}
Let $\xi\leq\alpha^{**}$ and suppose that $\l c_i \mid i<\lambda \r\in M_U$ be an increasing sequence, cofinal in $\kappa_{\xi}$.

Let $N'$ be a covering model for $\kappa_{\xi}$ with $\{c_i \mid i<\lambda \}\subseteq N'$.
Suppose that $\l c_i \mid i<\lambda \r$ are indiscernibles in $N'$ for $\kappa_{\xi}$. Then $c_i \in \{\kappa_\beta \mid \beta<\xi \}$, for all but finitely many $i$'s.
\end{lemma}
\begin{proof}
Similar to Lemma \ref{lemma: finite error}.
\end{proof} 

Now we can formulate the crucial property of subsets of $\kappa_{\xi}$ in $M_U$:
$(*)(B)$
\begin{enumerate}
\item $B\subseteq \kappa_{\xi} \cap \Lim(a^*)$.
%\item if $\tau \in B$, then $\tau$ is a limit point of $a^*$,
\item For every covering model $N'$ for $\kappa_{\xi}$  there is $\rho<\kappa_{\xi}$ such that for every indiscernible $c>\rho$ for $\kappa_{\xi}$ in $N'$, if $c$ is a limit point of $a^*$, then $c\in B$.
\item For every sequence $\l c_i \mid i<\theta\r\in M_U$ of elements of $B$, cofinal in $\kappa_{\xi}$, there is a covering model $N'$ for $\kappa_{\xi}$ and an ordinal $\theta'<\theta$ such that  $\l c_i \mid \theta'\leq i<\theta\r$ are indiscernibles for $\kappa_{\xi}$ in $N'$.
\end{enumerate}

\begin{lemma}\label{lem5-2-0-2}
$(*)(A(\xi))$ holds.
\end{lemma}
\begin{proof}
Requirement (1) is clear. Requirement (2) follows from Lemma \ref{lem5-2-0-1}. Let us show requirement (3). 

Indeed, for every function $g \colon \kappa_{\xi}^{<\omega} \to \kappa_\xi$ in $\mathcal{K}^{M_U}$ there is a function $f \in \mathcal{K}$ such that \[g(\bar x) = j(f)(\rho_0, \dots, \rho_{m-1}, \bar x),\] for some fixed $\rho_0, \dots, \rho_{m-1} < \kappa_\xi$. This follows from the iterated ultrapower representation. Thus, every $\kappa_{\alpha}$ which is larger than $\max(\rho_0, \dots, \rho_{m-1})$ would be a closure point of this function. In particular, taking $g$ to be the Skolem function of any covering model $N'$ for the sequence $\l c_i \mid i < \theta\r$ and taking $\theta'$ to be the least index in which $c_i > \max (\rho_0, \dots, \rho_{m-1})$. we conclude that each of the elements $c_i$ must be an indiscernible, by Theorem \ref{the covering lemma}.
\end{proof} 

\begin{lemma}\label{lem5-2-0-3}
If $(*)(B_1)$ and $(*)(B_2)$ hold, then $B_1$ agrees with $B_2$ on a final segment, i.e.\ there is $\nu<\kappa_{\xi}$ such that $B_1\setminus \nu=B_2\setminus \nu$.
\end{lemma}
\begin{proof}
Suppose otherwise.
By symmetry, let us assume that there is a cofinal in $\kappa_{\xi}$ sequence $\{e_i \mid i<\theta\} \in B_1\setminus B_2$, in $M_U$. By the first clause of $(*)(B_1)$, each $e_i$ is a limit point of $a^*$.

By $(*)(B_1)(3)$, there will be a covering model $N'$ for $\kappa_{\xi}$ with $\{e_i \mid i<\theta\} \subseteq N'$ such that for some $\theta'<\theta$,
$\l e_i \mid \theta'\leq i<\theta\r$ are indiscernibles for $\kappa_\xi$ in $N'$. Apply now $(*)(B_2)(2)$ to $N'$ and $\l e_i \mid \theta'\leq i<\theta\r$.
We will have then that a final segment of $\l e_i \mid \theta'\leq i<\theta\r$ is in $B$. Contradiction.
\end{proof} 

\begin{claim}\label{definable function}
If there is a function $t\in({}^{\kappa_\xi}\kappa_\xi)^{\calK^{M_U}}$ and some $\gamma<\kappa_\xi$ such that $A(\xi)\setminus\gamma=C_t\cap \Lim(a^*)\setminus\gamma$ then there is a uniformly definable function in $M_U$, $t^{\xi}\colon \kappa_\xi \to \kappa_\xi\in\calK^{M_U}$, with parameters $\kappa_\xi,a^*$, such that for some $\mu<\xi$, $(C_{t^{\xi}}\cap \Lim(a^*) \cap \kappa_\eta) \setminus\mu=A(\eta)\setminus\mu$.
\end{claim}
\begin{proof}
By assumption, $t$ satisfies the above equality, and by the previous claim, we let $t^{\xi}$ be the least function $t$ in the order of $\mathcal{K}^{M_U}$ such that $(*)(C_{t}\cap \Lim(a^*))$ holds. This is formulated in $M_U$ using the parameters $\kappa_\xi$ and $a^*$.\end{proof}
  
Back to $\kappa_{\eta_i}$'s, by the induction hypothesis and by claim \ref{definable function}, fix the function $t^{\eta_i}$ which is definable with parameters $\kappa_{\eta_i}, a^*$.
  
\begin{lemma}\label{lem5-2-1}
Assume that $N_0$ is a covering model for $\kappa_\eta$ and $h^{N_0}\in\calK^{M_U}$ the associated Skolem function. Consider $\tilde{h}^{N_0}\colon \kappa_\eta\to \kappa_\eta\in\calK^{M_U}$ defined as follows:
$$\tilde{h}^{N_0}({\rho})=\sup(\{ h^{N_0}(\vec{\xi})\mid \vec{\xi}\in [\rho+1]^{<\omega} \text{ and } h^{N_0}(\vec{\xi})<\kappa^* \}).$$
Suppose $\eta'<\eta$ is such that $\kappa_{\eta'},a^*\in N_0$ and $t^{\eta'}$ is definable as above.
\\Then for all but boundedly many  $\nu<\kappa_{\eta'},$
 $\tilde{h}^{N_0}(\nu)\geq t^{\eta'}(\nu)$.
\end{lemma}

\begin{proof}
We use the elementarity of $N_0$ and the definability of $t^{\eta'}$ to conclude that $t^{\eta'} \in N_0\cap \calK^{M_U}$. Note that
$t^{\eta'}= h^{N_0}(\vec{c})$, for a finite sequence of $N_0$-indiscernibles $\vec{c}$ $\leq \kappa_{\eta'}$. By the construction of the covering model $N_0$, we can find $t\in N_0$, $t\colon \kappa_\eta \to \kappa_\eta$ such that $t\restriction\kappa_{\eta'}=t^{\eta'}$ and $t=h^N(\vec{c}')$ where $\vec{c}'$ are all indiscernables strictly below $\kappa_{\eta'}$. 

Hence by the definition of $\tilde{h}^{N_0}$, for every $\max(\vec{c}')\leq\nu<\kappa_{\eta'}$, 
$t^{\eta'}(\nu)\leq \tilde{h}^{N_0}(\nu)$.
\\It follows then by the definition of $\tilde{h}^{N_0}$ that for all but boundedly many  $\nu<\kappa_{\eta'},$
 $\tilde{h}^{N_0}(\nu)\geq t^{\eta'}(\nu)$.
\end{proof}

For every $i<\lambda$, apply lemma \ref{lem5-2-1} to $\kappa_{\eta_i}$ and the model $N$ to find $\nu_i<\kappa_{\eta_i}$ such that for every $\nu_i\leq \nu<\kappa_{\eta_i}$,
$\tilde{h}^N(\nu)\geq t^{\eta_i}(\nu)$.
\\Then, by pressing down, and since $\lambda>\omega$, there will be a stationary $Z\subseteq \lambda$ and $\nu^*<\lambda$ such that for every $\nu, \nu^*\leq \nu<\kappa_{\eta_\xi}$, $\xi\in Z$ the inequality $\tilde{h}^N(\nu)\geq t^{\eta_\xi}(\nu)$ holds.

Now, shrinking $Z$ more if necessary, we will get $\nu^{**}<\kappa^*$ such that
$$C_{\tilde{h}^N}\cap \Lim(a^*)\setminus \nu^{**}=A(\eta)\setminus \nu^{**}.$$
\vskip 0.5 cm

\textbf{Case 2:} Suppose that $\lambda=\omega$.%\footnote{For example if $o^{\calK^{M_U}}(\kappa^*)=\kappa^*$ then $\lambda=\omega$, this is a situation which was not dealt with in the previous section.}.

Once again, since $\kappa_{\eta}\in a^*$ we can find an increasing %continuous 
and cofinal sequence in $\kappa_\eta$, $\l\kappa_{\eta_n}\mid n<\omega\r\in M_U$. %such that $\eta_0\geq\gamma_0$, meaning that $o^{\calK^{M_U}}(\kappa_{\eta_n})<(\kappa_{\eta_n}^+)^{M_U}$. 
Let us add points to this sequence.
If $\kappa_{\eta_n}\in \Lim(a^*)$, apply the induction hypothesis, find $t^{\eta_n}$ and let $\nu_n<\kappa_{\eta_n}$ be minimal such that $$C_{t^{\eta_n}}\cap \Lim(a^*)\setminus\nu_n= A(\eta_n)\setminus \nu_n.$$ 
Find $\xi_n<l^*$ be such that $\crit(j_{\xi_n,\xi_{n}+1})=\kappa_{\eta_n}$, then $t^{\eta_n}\in\calK_{\xi_n}$.
We can represent $t^{\eta_n}$ in the iteration using some $f_n\in({}^{\kappa}\kappa)^{\calK}$ and some intermediate critical points $\lambda_1,\dots ,\lambda_m<\kappa_{\eta_n}$, $j_{\xi_n}(f_n)(\lambda_1,\dots,\lambda_m)=t^{\eta_n}$. 
Let $$\max((\{\kappa_\alpha\mid \eta_{n-1}<\alpha<\eta_n\}\cap\{\lambda_1,\dots,\lambda_m\})\cup\{\kappa_{\eta_{n-1}}\})=\kappa_{\eta_{n,1}}$$
By minimality of $\nu_n$, there is $\eta_{n,2}<\eta_n$ such that $\kappa_{\eta_{n,2}}\leq \nu_n\leq \kappa_{\eta_{n,2}+1}$. If $\eta_{n,2}\leq\eta_{n,1}$ then add $\eta_{n,1}$ to the sequence and set $\eta^{(1)}=\eta_{n,1}$. Otherwise, add $\kappa_{\eta_{n,2}}$ to the sequence and set $\eta^{(1)}=\kappa_{\eta_{n,2}}$.
If $\kappa_{\eta_n}\notin \Lim(a^*)$, denote by $$\nu_n=\sup(a^*\cap \kappa_{\eta_n})<\kappa_{\eta_n}$$ There is $\eta'<\eta_n$ such that $\kappa_{\eta'}\leq \nu_n<\kappa_{\eta'+1}$ and there is a function $t^{\eta_n}\in({}^{\kappa_{\eta_n}}\kappa_{\eta_n})^{\calK_{\xi_n}}$ such that $\nu_n\leq t^{\eta_n}(\kappa_{\eta'})$. Indeed, by lemma \ref{BoundFun}, there is $f\in({}^{\kappa}\kappa)^{\calK}$ such that $j_{\xi+1}(f)(\kappa_{\eta'})\geq \nu_n$, where $\xi<\xi_n$ is the step of the iteration such that $\kappa_{\eta'}$ is a critical point. Then we can set $t^{\eta_n}=j_{\xi_n}(f)$.  Let $\eta^{(1)}=\eta'$. 

In any case, if $\eta^{(1)}\leq \eta_{n-1}$ we are done. Otherwise, we move to $\kappa_{\eta^{(1)}}$ and repeat the above. After finitely many steps, defining $\eta^{(k)}<\eta^{(k-1)}<\cdots<\eta^{(1)}<\eta_n$ we reach $\eta_{n-1}$. After adding these new points, we obtain a sequence still of order type $\omega$. Without loss of generality, this was the sequence $\l \kappa_{\eta_n}\mid n<\omega\r$ that we started with. During the construction we have defined a sequence of functions $\l t^{\eta_n}\mid n<\omega\r$, such that $t^{\eta_n}\in ({}^{\kappa_{\eta_n}}\kappa_{\eta_n})^{\calK^{M_U}}$ and by closure $\l t^{\eta_n}\mid n<\omega\r\in M_U$.
 Clearly, $t^{\eta_n}\in \calK_{\xi_n}$. Let $\xi^*=\sup\xi_n$, then $\crit(j_{\xi^*,l^*})\geq\kappa_\eta$.

\begin{claim}\label{dominating claim}
There is $\phi\in({}^{\kappa}\kappa)^\calK$ such that $\forall n<\omega$, and every $\kappa_{\eta_{n-1}}\leq\nu<\kappa_{\eta_n}$ $t^{\eta_n}(\nu)<j_{\xi_n}(\phi)(\nu)$
\end{claim}
\begin{proof}
By construction of the sequence $\l \kappa_{\eta_n}\mid n<\omega\r$, either $\kappa_{\eta_n}\notin \Lim(a^*)$ in which case there is $f_n\in\calK$  such that $t^{\eta_n}=j_{\xi_n}(f_n)$ (no parameters needed). If $\kappa_{\eta_n}\in \Lim(a^*)$, then by the construction of the sequence $\kappa_{\eta_n}$, there is a function $f_n\in\calK$ and critical points $$\lambda_1<\cdots<\lambda_k<\kappa_{\eta_{n-1}}<\theta_1<\cdots<\theta_m<\kappa_{\eta_{n-1}+1}\leq \kappa_{\eta_n}$$ such that
$$t^{\eta_n}=j_{\xi_n}(f_n)(\lambda_1,\dots,\lambda_k,\kappa_{\eta_{n-1}},\theta_1,..,\theta_m).$$ 
Since $\theta_m<\kappa_{\eta_{n-1}+1}$,
by lemma \ref{BoundFun}, there is $b_n\in({}^{\kappa}\kappa)^{\calK}$ such that $$\theta_m<j_{\xi_{n-1}}(b_n)(\kappa_{\eta_{n-1}})\leq j_{\xi_{n}}(b_n)(\kappa_{\eta_{n-1}})\leq j_U(b_n)(\kappa_{\eta_{n-1}})<\kappa_{\eta_{n}}.$$ 
In $\calK$, define $\phi_n:\kappa\rightarrow\kappa$ by $$\phi_n(\alpha)=\sup\{f_n(\vec{\rho})(\xi)\mid\vec{\rho}\in[b_n(\alpha)]^{<\omega}\cap \dom(f_n)\wedge \xi\leq\alpha\}+1.$$
Then for every $\kappa_{\eta_{n-1}}\leq\nu<\kappa_{\xi_n}$,
$$j_{\xi_n}(f_n)(\vec{\lambda},\kappa_{\eta_{n-1}},\vec{\theta})(\nu)\leq \sup\{j_{\xi_n}(f_n)(\vec{\xi})(\xi)\mid\vec{\xi}\in[j_{\xi_n}(b_n)(\nu)]^{<\omega}\wedge \xi\leq\nu\}.$$
Hence $t^{\eta_n}(\nu)=j_{\xi_n}(f_n)(\vec{\lambda},\kappa_{\eta_{n-1}},\vec{\theta})(\nu)<j_{\xi_n}(\phi_n)(\nu)$.
We proceed as in lemma \ref{lemma: finite error}. Suppose that $\l d_i\mid i<\kappa^+\r$ is an enumeration of $({}^{\kappa}\kappa)^{\calK}$ and that $\phi_n=d_{\mu_n}$ There is a set $a\subseteq\kappa^+$ such that $a\in\calK, \ |a|<\kappa$ and $\{\mu_n\mid n<\omega\}\subseteq a$. Define in $\calK$, $\phi:\kappa\rightarrow\kappa$ by
$$\phi(\alpha)=\sup\{d_i(\alpha)\mid i\in a\}$$
Since $\kappa$ is regular in $\calK$, $\phi$ is well defined 
and for every $n<\omega$, $\phi$ dominates $\phi_n$ everywhere. By elementarity of $j_{\xi_n}$, $\phi$ will be as desired
\end{proof}

Denote by $t^\eta=j_{U}(\phi)\restriction\kappa_{\eta}\in\calK^{M_U}$. Note that $t^\eta\restriction\kappa_{\eta_n}=j_{\xi_n}(\phi)$. Let us prove that $t^\eta$ is as wanted:

\begin{claim}
There is $\gamma_\eta<\kappa_\eta$ such that
$$(C_{t^\eta}\cap \Lim(a^*) \cap \kappa_\eta)\setminus \gamma_{\eta}=A(\eta)\setminus\gamma_{\eta}.$$
\end{claim}
\begin{proof}
As we claimed before, $\{\kappa_\gamma\mid \gamma_\eta\leq \gamma<\eta\}$ is a weak Prikry-Magidor sequence for $\calK^{M_U}$ and $C_{t^\eta}$ is a club in $\calK^{M_U}$, there is $\gamma_\eta$ such that $\{\kappa_\gamma\mid \gamma_\eta\leq \gamma<\eta\}\subseteq C_{t^{\eta}}$. This proves 
the inclusion from right to left. For the other direction, assume that $\delta\in C_{t^\eta}\setminus\kappa_{\gamma_\eta}$ such that $\delta\notin\{\kappa_\gamma\mid \gamma_\eta\leq\gamma<\eta\}$, let us argue that $a^*\cap \delta$ is bounded below $\delta$. Fix any $n<\omega$ such that $\kappa_{\eta_n}<\delta<\kappa_{\eta_{n+1}}$.  We split into cases. If $\kappa_{\eta_{n+1}}\notin \Lim(a^*)$, then
$$\sup(a^*\cap\delta)\leq\sup(a^*\cap\kappa_{\eta_{n+1}})=\nu_n\leq t^{\eta_{n+1}}(\kappa_{\eta_n})$$ 
By claim \ref{dominating claim}, $t^{\eta_{n+1}}(\kappa_{\eta_n})<j_{\xi_n}(\phi)(\kappa_{\eta_n})=t_\eta(\kappa_{\eta_n})$.
Since $\kappa_{\eta_n}<\delta\in C_{t_\eta}$, we conclude that $\sup(a^*\cap\delta)<\delta$ and $\delta$ is not a limit point of $a^*$. 

If $\kappa_{\eta_{n+1}}\in \Lim(a^*)$, then by the construction of $\kappa_{\eta_n}$ we have that $C_{t^{\eta_{n+1}}}\cap \Lim(a^*)\setminus \kappa_{\eta_n}=A(\eta_{n+1})\setminus \kappa_{\eta_n}$.
By assumption, $\delta\notin \{\kappa_\alpha\mid \gamma_\eta\leq \alpha<\eta\}$, hence $\delta\notin A(\eta_{n+1})$. Since $\kappa_{\eta_n}<\delta$, it follows that $\delta\notin C_{t^{\eta_{n+1}}}\cap \Lim(a^*)$.
\end{proof}
This conclude that proof of lemma \ref{lem:isolating-the-indiscernibles}, and the proof of
theorem \ref{thm:lowerbound}.\end{proof}

It is possible to try to proceed further and to deal with the situation when $o(\kappa^*) = (\kappa^*)^+$. If, as a result, $\kappa^*$ remain regular (which is typical forcing situation) then $a^*$ must be bounded in $\kappa^*$, since no regular cardinal can be in $a^*$, and so we are basically   in the situation considered above.
\\However, $\kappa^*$ can change cofinality --- there are forcing construction in which it changes cofinality to $\omega$.
In this case a finer analysis of indiscernibles seems to be needed, and Mitchell's accumulation points may appear.

Our conjecture is that the result above is not optimal and it can be strengthened.
\section{Compactness for masterable forcing notions}

In this section we will isolate a subclass of forcing notions that consistently include many important forcing notions (such as all the complete subforcings of $\Add(\kappa,1)$ and more), such that it is possible to force from a measurable cardinal that for any forcing $\mathbb{P}$ in this class, there is a $\kappa$-complete ultrafilter extending $\mathcal{D}(\mathbb{P})$.

\begin{lemma}\label{lemma}
Let $\mathbb{Q}$ be a $\kappa$-distributive forcing of size $\kappa$.
\\Suppose that there is a generic elementary embedding
\[j\colon V^{\mathbb{Q}} \to M\]
with $\crit j = \kappa$.
 Then, in $M$, there is a single condition $m\in j(\mathbb{Q})$ which is stronger than $j(p)$ for any condition $p$ in the generic filter for $\mathbb{Q}$.
\end{lemma}
\begin{proof}

Without loss of generality we can assume that $\mathbb{Q}=\kappa$, i.e.\ the set of conditions of the forcing $\mathbb{Q}$ is just $\kappa$.
Let $G\subseteq\mathbb{Q}$ be the generic filter. By elementarity, $M = M'[j(G)]$, where $\forall p\in G,\,j(p)\in j(G)$.
Note that since $G\subseteq\kappa$ and $\crit(j)=\kappa$, $G=j(G)\cap\kappa\in M$.
%Since $|\mathbb{Q}| = \kappa$, $\mathbb{Q}=j\image \mathbb{Q} \in
% {^\kappa} M\cap V^{\mathbb{Q}}$ and thus in $M$.
By the distributivity of $\mathbb{Q}$ over $V$ and by elementarity of $j$, $j(\mathbb{Q})$ is also $j(\kappa)$-distributive over $M'$, hence $G\in M'$. In particular, the set
$$D=\{q \in j(\mathbb{Q}) \mid ((\forall p \in G)(q \geq j(p)))\vee ((\exists p \in G)(q\perp j(p)))\}$$
 is dense open in $M'$. Clearly, any condition $m\in j(G)$ from this set will witness the validity of the lemma, since $j(G)\supseteq j\image G=G$.
\end{proof}
Define now a subclass of $\kappa$-distributive forcing of size $\kappa$.

\begin{definition}\label{definition: masterable}
A forcing notion $\mathbb{Q}$ is called \emph{masterable} if

\begin{enumerate}
  \item $\mathbb{Q}$ is a $\kappa$-distributive forcing of size $\kappa$,
  \item there is a forcing notion $\lusim{\mathbb{R}} \in V^Q$ such that
 \begin{enumerate}
   \item In $V^{\mathbb{Q}*\lusim{\mathbb{R}}}$, there is an elementary embedding
   $$j\colon V^{\mathbb{Q}} \to M$$
with $\crit j = \kappa$.
   \item $\mathbb{Q}*\lusim{\mathbb{R}}$ contains a dense subset of size $\leq \kappa$
   and $\mathbb{Q} \ast \name{\mathbb{R}}$ is ${<}\kappa$-strategically closed.
 \end{enumerate}

\end{enumerate}
\end{definition}

 Let $\mathcal{N}_\kappa$ denotes the class of all masterable forcing notions.

\begin{claim}
$\mathcal{N}_\kappa$ is closed under complete subforcings.
\end{claim}
\begin{proof}
Assume $\mathbb{Q}$ is a complete subforcing of $\mathbb{P}\in \mathcal{N}_\kappa$.
Then $|\mathbb{Q}|\leq |\mathbb{P}|\leq\kappa$ and let $\name{\mathbb{R}}$ witness propery $(2)$ for $\mathbb{P}$. Let $\name{\mathbb{R}'}=\mathbb{P}/\name{G_{\mathbb{Q}}}\ast\name{\mathbb{R}}$ where $\mathbb{P}/\name{G_{\mathbb{Q}}}$ is then quotient forcing. Now $\mathbb{Q}*\name{\mathbb{R}'}\simeq \mathbb{P}\ast\name{\mathbb{R}}$ and so condition $(2)$ holds for $\mathbb{Q}$.
\end{proof}

\begin{theorem}\label{thm1}
Assume GCH and let $\kappa$ be a measurable cardinal.

Then there is a cofinality preserving forcing extension in which for any $\mathbb{Q}\in\mathcal{N}_\kappa$, there is a $\kappa$-complete ultrafilter $\mathcal{U}$ extending $\mathcal{D}_p(\mathbb{Q})$ for every $p\in \mathbb{Q}$.

\end{theorem}

\begin{proof}
Let $\mathbb{P}_\kappa$ be a Easton support iteration of length $\kappa$, $\langle \mathbb{P}_\alpha, \lusim{\mathbb{Q}}_\beta \mid \alpha\leq \kappa, \beta<\kappa \rangle$. At each step, $\name{\mathbb{Q}}_\alpha$ is either the trivial forcing, if $\alpha$ is not inaccessible, or the lottery sum of all ${<}\alpha$-strategically closed forcing notions of size $\alpha$ (were the trivial forcing is included).

Let $G_\kappa\subseteq P_\kappa$ be a generic. We argue that the model $V[G_\kappa]$ is as desired.

Let $\mathbb{Q}$ be a forcing notion in $\big(\mathcal{N}_\kappa\big)^{V[G_\kappa]}$ and $p\in\mathbb{Q}$. Let $U$ be a normal, $\kappa$-complete ultrafilter over $\kappa$. Let $j_1 \colon V \to N_1 \cong \Ult(V, U)$ be the ultrapower maps using $U$. Let $\kappa_1 = j_1(\kappa)$.

Let us extend, in $V[G_\kappa]$, the embedding $j_1$ to an elementary embedding
\[j_1^*:V[G_\kappa]\to N_1[G_{\kappa_1}].\]
Indeed, $j(\mathbb{P}_\kappa) = \mathbb{P}_\kappa \ast j(\mathbb{P})_{[\kappa, j(\kappa))}$. By picking the trivial forcing at $\kappa$, the rest of the iteration is $\kappa^{+}$-strategically closed in $V$ (by the closure of $N_1$ to $\kappa$-sequences). The number of dense open sets of the tail forcing is $\kappa^{+}$ (as enumerated in $V$) and thus one can construct in $V[G_\kappa]$ an $N_1[G_\kappa]$-generic filter for the tail forcing $j(\mathbb{P})_{\kappa,j(\kappa)}$. Let $G_{\kappa_1}$ be the generic filter for $N_1$.

By elementarity, $j_1^*(\mathbb{Q}) \in (\mathcal{N}_{\kappa_1})^{N_1[G_{\kappa_1}]}$, thus by condition $(2)$ there is a $\name{\mathbb{R}}$ and a dense subset $X\subseteq j_1^*(\mathbb{Q})*\name{R}$ such that $N_1[G_{\kappa_1}]\models|X|\leq \kappa_1$. By $GCH$, from the point of view of $V[G_\kappa]$, there are $\kappa^+$ dense open sets to meet in order to generate a generic filter for $j_1^*(\mathbb{Q})$. By condition $(2)$, $j_1^*(\mathbb{Q})*\name{\mathbb{R}}$ is ${<}\kappa_1$-strategically closed in $N_1[G_{\kappa_1}]$, again by closure of $N_1[G_{\kappa_1}]$ to $\kappa$ sequences from $V[G_k]$, it is $\kappa^+$-strategically closed from the point of view of $V[G_\kappa]$. Hence, one can find a $N_1[G_{\kappa_1}]$-generic filter, $G_{j_1^*(\mathbb{Q})}*G_{\name{\mathbb{R}}} \in V[G_\kappa]$ with $j_1^*(p)\in G_{j_1^*(\mathbb{Q})}$.
Since $j_1^*(\mathbb{Q})$ is masterable using the forcing $\name{\mathbb{R}}$ in the extension $N_1[G_{\kappa_1}][G_{j_1^*(\mathbb{Q})}*G_{\name{\mathbb{R}}}]$ there is an elementary embedding
\[k\colon N_1[G_{\kappa_1}][G_{j_1^*(\mathbb{Q})}] \to N^*\]
such that $\crit k = \kappa_1$.
Let $m$ be a condition in $k(j_1(\mathbb{Q}))$ such that $m$ is stronger than $k(p)$ for all $p\in G_{j_1^*(\mathbb{Q})}$ which exists by applying lemma \ref{lemma} to $j_1^*(\mathbb{Q})$. In $V[G_\kappa]$, define $$\mathcal{U} = \{A\subseteq \mathbb{Q}\mid m \in k(j^*_1(A_\xi))\}$$ It is clear that $\mathcal{U}$ is a $\kappa$-complete ultrafilter that extends $\mathcal{D}_p(\mathbb{Q})$.
\end{proof}
\begin{corollary}\label{MasterForcing}
Consider $\mathcal{N}_\kappa$ of the model of the previous theorem $V[G_\kappa]$. Then
\begin{enumerate}
  \item $\Add(\kappa,1)\in\mathcal{N}_\kappa$, and hence, by the claim above,  all its complete subforcings are in $\mathcal{N}_\kappa$
  ( for example:  adding a Suslin tree to $\kappa$, adding a non-reflecting stationary subset of a given stationary set etc.).
  \item $Club(S)\in\mathcal{N}_\kappa$ for all $S\subseteq \kappa$ that contains all the singular cardinals and is of measure one in a normal measure over $\kappa$.
\end{enumerate}

\end{corollary}
\begin{proof}

For (1), we wish to prove that $\mathbb{Q}=\Add(\kappa,1)\in (N_\kappa)^{V[G_{\kappa}]}$.
Let $f$ be $V[G_\kappa]$-generic for $\Add(\kappa,1)$, We will extend in $V[G_\kappa][f]$ the elementary embedding $j_U:V\rightarrow M_U$ to $$j^*:V[G_\kappa][f]\rightarrow M_U[G_{\kappa_1}][f']$$
Then we can take $\lusim{\mathbb{R}}$ to be the trivial forcing in the definition of masterable.
the generic $G_{\kappa_1}$ will be made of $G_\kappa$ followed by $f$ as generic for $Q_\kappa$, then a $M[G_\kappa*f]$-generic filter for the rest of the forcing $P_{(\kappa,\kappa_1]}$, can be constructed in $V[G_k][f]$ using the strategic closure of of the forcing as we did in theorem \ref{thm1}. Also we can find the generic $f_{\kappa_1}\in V[G_\kappa][f]$ for $(\Add(\kappa_1,1))^{M[G_{\kappa_1}]}$, and  $f_{\kappa_1}\restriction\kappa=f$. Note that this is a condition in $\Add(\kappa_1,1)^{M[G_{\kappa_1}]}$. Above this condition, we can construct the generic $f_{\kappa_1}$ since again $\Add(\kappa_1,1)^{M[G_{\kappa_1}]}$ as $\kappa^+$ many dense open subsets from the point of view of $V[G_\kappa][f]$ and is it $k^+$-closed since the model is closed under $\kappa$-sequences.

For (2),
 Let $S\subset\kappa$ be a stationary set that contains all singular cardinals and let us assume that $S\in W$, for normal measure $W$ over $\kappa$.
  We need to show that $Club(S) \in (\mathcal{N}_\kappa)^{V[G_\kappa]}$. Indeed, the forcing $Club(S)$ is ${<}\kappa$-strategically closed. Let $H\subseteq Club(S)$ be $V[G_\kappa]$-generic.
Let us show that in $V[G_\kappa][H]$, the elementary embedding $j_W$, which corresponds to $W$, extends to an elementary embedding:
\[j_W^{\prime}\colon V[G_\kappa] \to N_W[G_{\kappa_1}'],\]
where $\kappa_1=j_W(\kappa)$,
by taking the generic of $j(\mathbb{P}_\kappa) \restriction \kappa + 1$ to be $G_\kappa \ast H$ and extending it to a generic filter $G_{\kappa_1}'$, using the $\kappa^{+}$-strategically closure of the tail forcing in $V[G_\kappa][H]$.

Since $\crit j'_W = \kappa$, for any $p\in Club(S)$, $j'_W(p) = p$. Also, since $\kappa \in j'_W(S)$.
\[m = \{\kappa\}\cup\bigcup_{p\in H} p \in j_W'(Club(S)).\]
Using the same arguments as before, we can find an $N_W[G_{\kappa_1}]$-generic filter $H'\in V[G_\kappa][H]$ for $j_W'(Club(S))$ such that $m\in H'$. We conclude that the embedding $j_1'$ extends to an embedding:
\[j_W''\colon V[G_\kappa][H]\to N_1[G_{\kappa_1}][H'].\]
Therefore, we can take $\name{\mathbb{R}}$ to be the trivial forcing.
\end{proof}

Note that in general $\Add(\kappa,1)$ might not be masterable. For example, if we force above $L[U]$ with $\Add(\kappa,1)$ the $\kappa$ is no longer measurable.

Let us deduce now one more corollary that relates to the result of section 6.

\begin{corollary}
Consider, in $V[G_\kappa]$, the forcing for adding a club through singulars and inaccessibles which are not Mahlo, i.e.
 $$\mathbb{Q}=\{a \subseteq \kappa \mid |a|<\kappa, a \text{ is closed and each member of } $$$$a \text{ is either a singular cardinal or an inaccessible which is not a Mahlo}\} $$
 ordered by end-extension.
Then $\mathbb{Q}\in \mathcal{N}_\kappa$.

\end{corollary}
\begin{proof}
Let $G(\mathbb{Q})$ be a $V[G_\kappa]$-generic subset of $\mathbb{Q}$.
Clearly, $\mathbb{Q}$ is a ${<}\kappa$-strategically closed forcing of cardinality $\kappa$.
Let $\name{\mathbb{R}}$ be the forcing for adding a club through singulars over $V[G_\kappa,G(\mathbb{Q})]$.
Again $\mathbb{Q}*\name{\mathbb{R}}$ is a ${<}\kappa$-strategically closed forcing of cardinality $\kappa$.

Let $G(\name{\mathbb{R}})$  be a generic subset of $\name{\mathbb{R}}$ over $V[G_\kappa,G(\mathbb{Q})]$.
We shall argue that in $V[G_\kappa,G(\mathbb{Q}), G(\name{\mathbb{R}})]$ there is an elementary embedding
$$i\colon V[G_\kappa,G(\mathbb{Q})] \to M,$$
with $\crit(i) = \kappa$ and $({}^\kappa M )\cap  V[G_\kappa,G(\mathbb{Q})]\subseteq M$.

Let $U$ be a normal ultrafilter over $\kappa$ in $V$ and
$j:V\to N$ the corresponding elementary embedding.
Work in $V[G_\kappa,G(\mathbb{Q}), G(\name{\mathbb{R}})]$ and extend it to an elementary embedding
$$i:V[G_\kappa,G(\mathbb{Q})] \to N[G_{j(\kappa)}, G(j(\mathbb{Q}))]$$
as follows.
Set $ G_{j(\kappa)}\restriction \kappa=G_\kappa$.
Now let $Q_\kappa=\mathbb{Q}*\name{\mathbb{R}}$ and take $G(\mathbb{Q})* G(\name{\mathbb{R}})$ to be its generic subset.

Note that $\kappa$ was a Mahlo cardinal in $ V[G_\kappa,G(\mathbb{Q})]$, and hence, in $N[G_\kappa,G(\mathbb{Q})]$, but $G(\mathbb{R})$ destroys its Mahloness.
We complete building  $G_{j(\kappa)}$ using the strategic closure of the relevant forcing.

Let $G(j(\mathbb{Q}))$ starts with $\bigcup G(\mathbb{Q}) \cup \{\kappa\}$. $\kappa$ is not Mahlo anymore, and so, can be added.
Finally, complete building  $G_{j(\mathbb{Q})}$ using the strategic closure of the  forcing $j(\mathbb{Q})$ i.e. we have $\kappa^+$ many dense open sets to meet, the bad player starts with playing  $G(\mathbb{Q}) \cup \{\kappa\}$ and then using the strategy we meet the rest of the dense open sets.

This completes the proof of $\mathbb{Q} \in \mathcal{N}_\kappa$.
\end{proof}
\section{Other examples}
The next interesting examples should be of forcings of size $\kappa$, which are $\kappa$-distributive,  but not ${<}\kappa$-strategically closed nor masterable.

Let start with two simple general observations.

\begin{prop}
Let $\kappa>\aleph_1, \eta<\kappa$ be a regular cardinals. Assume that for every $\lambda<\kappa$, $\lambda^{<\eta}<\kappa$.
Suppose that $\l Q,\leq_Q\r $ is an $\eta+1$-strategicaly closed forcing notion.
Then $\l Q, \leq_Q\r$ preserves stationary subsets of $\kappa$ which concentrate on cofinality $\eta$ i.e. For any set $S$ such that $S\subset\{\nu<\kappa \mid \cof(\nu)=\eta\}$ is stationary, $\Vdash_Q \dot{S}$ is stationary.
\end{prop}
\begin{proof}
  Let $S\subseteq \{\nu<\kappa \mid \cof(\nu)=\eta\}$ be stationary.
Suppose that for some generic subset $G(Q)$ of $Q$, $S$ is non-stationary in $V[G(Q)]$.
Let $C\subseteq \kappa$ be a club  disjoint from $S$. Let $\name{C}$ be a $Q$-name for $C$.

Then, back in $V$ there are $q \in G(Q)$ such that
$$q \Vdash (\lusim{C}\subseteq \kappa \text{ is a club and } S\cap \lusim{C}=\emptyset).$$
Fix a winning strategy $\sigma$ for the Player I in  plays of the length $\eta+1$ for $Q$.
\\Pick now an elementary submodel  $N$ of $H_\theta$, with $\theta$ large enough, such that
 \begin{enumerate}
   \item $N\supseteq \eta+1$ and $Q, S, \sigma, \lusim{C},q \in N$.
   \item $\kappa>|N|\geq\eta$,
   \item $\sup(N\cap \kappa)\in S$,
   \item ${}^{<\eta}N\subseteq N$,

 \end{enumerate}
 This is possible since we can construct an continuous and increasing sequence of models $\langle N_i\mid i<\eta\rangle$ satisfying $(1),(2)$, ${}^{<\eta}N_i\subseteq N_{i+1}$ and $\sup(N_i\cap\kappa)<\kappa$. Since $\eta$ is regular and $\eta^{<\eta}=\eta$ we can construct such a sequence and $\cup_{i<\eta}N_i=N^*_0$. Then $N^*_0$ satisfy $(1),(2),(4)$. We keep defining increasing and continuous models $$\langle N^*_i\mid i<\kappa\rangle$$ satisfying $(1),(2)$ and at successor points also $(4)$. In this definition we exploit the cardinal assumption that for every $\lambda<\kappa$, $\lambda^<\eta<\kappa$. The set $$\{\sup(N^*_i\cap\kappa)\mid i<\kappa\}$$ is a club at $\kappa$ thus there is $i<\kappa$ such that $\alpha=\sup(N^*_i\cap\kappa)\in S$. Note that the cofinality of $\alpha$ is $\eta$ and therefore $(N^*_i)^{<\eta}\subseteq N^*_i$. Let $N=N^*_i$, then $N$ satisfy $(1)-(4)$.

Let $\l \xi_i \mid i<\eta \r$ be a cofinal sequence in $\sup(N\cap \kappa)$. By $(4)$, every initial segment of it is in $N$.

Using $\sigma$ it is easy to  define an increasing sequence of conditions $\l q_i \mid i\leq \eta\r$ in $Q$ such that
 \begin{enumerate}
   \item $q_0=q$.
   \item $q_i \in N$, for every $i<\eta$.
   \item There is $\alpha_i\geq\xi_i$ such that $q_{i+1}\Vdash\dot{\alpha_i}\in\lusim{C}$.
 \end{enumerate}
  Since $q_i\in N$, $\alpha_i\in N\cap\kappa$ such $\langle\alpha_i\mid i<\eta\rangle$ form an increasing and continuous sequence in $\alpha$. Let $p_\eta=\sigma(\langle p_i q_i\mid i<\eta\rangle)$,
 then $$q_\eta\Vdash  \sup(N\cap \kappa)\in \lusim{C},$$ since it also forces that $\lusim{C}$ is closed.

 This is impossible, since $\sup(N\cap \kappa)\in S$, Contradiction.
 \end{proof}

\begin{prop}
Let $\kappa>\aleph_1, \eta<\kappa$ be a regular cardinals. Assume that for every $\lambda<\kappa$, $\lambda^{<\eta}<\kappa$.
Suppose that $\l \mathbb{P},\leq_{\mathbb{P}}\r$ is a forcing notion that destroys stationarity of a subset of $\kappa$ which concentrate on cofinality $\eta$.
Then $P$ is not masterable.

\end{prop}
\begin{proof}
 Suppose otherwise.
 Then there is a forcing notion $\lusim{\mathbb{R}}$ such that $\mathbb{P}*\lusim{\mathbb{R}}$ is ${<}\kappa$-strategically closed.
 In particular,  $\mathbb{P}*\lusim{\mathbb{R}}$ is $\eta+1$-strategically closed.

 By the previous proposition, then  $\mathbb{P}*\lusim{\mathbb{R}}$ preserves stationary subsets of $\kappa$ which concentrate on cofinality $\eta$.

 This is impossible since already  $\l \mathbb{P},\leq_{\mathbb{P}}\r$ is a forcing notion that destroys stationarity of some stationary subset $S\subseteq \kappa$ which concentrate on cofinality $\eta$, hence the witnessing club which is disjoint from $S$ will also be present in extensions of  $\mathbb{P}*\lusim{\mathbb{R}}$, Contradiction.
 \end{proof}

Now we deal with a particular example.
Let $S$ be a fat subset of $\kappa$ such that $$ \{\nu<\kappa \mid \cof(\nu)=\eta\}\setminus S$$ is stationary.

Then, the forcing $Club(S)$ is $\kappa$-distributive (since $S$ is fat). $Club(S)$ shoots a club through $S$ and therefore distroys the stationarity of $\{\nu<\kappa\mid \cof(\nu)=\eta\}\setminus S$. It follows that $Club(S)$ is not $<-\kappa$-strateginaly closed  (even not $\eta+1$-strateginaly closed ) and not masterable.

%\begin{question}
%Does the method used for masterable forcings works for $Club(S)$ as well?
%\end{question}
Note that if we force a Cohen function $f:\kappa \to \kappa$, then for every $\delta<\kappa$ the set
$$S^f_\delta=\{\nu<\kappa \mid f(\nu)=\delta\}$$ will be a fat stationary subset of $\kappa$ such that
for every regular $\eta<\kappa$,
the set
$$S^f_\delta\cap \{\nu<\kappa \mid \cof(\nu)=\eta\}$$ is co-stationary. The next lemma shows that a similar method to the one used for masterable forcings, can be used to extend $D_p(Q)$ for this kind of fat stationary sets.
\begin{lemma}
Let $\kappa$ be measurable cardinal and assume $GCH$. There is a cofinality preserving extension $V[G_\kappa]$ in which the following holds:

After forcing a Cohen function $f:\kappa\to \kappa$ with $\Add(\kappa,1)^{V[G_\kappa]}$, for every $\delta<\kappa$ and $p\in Club(S^{f}_\delta)$, $D_p(Club(S^{f}_{\delta}))$ can be extended to a $\kappa$-complete ultrafilter.
\end{lemma}
\begin{proof}
 Let us use the same Easton support iteration $\langle P_\alpha,\name{Q}_\beta\mid \alpha\leq\kappa,\beta<\kappa\r$ as for masterable forcing, where $\name{Q_\beta}$ is the trivial forcing for accessible ordinals and the lottery sum over all ${<}\beta$-strategically closed forcings of size $\leq\beta$ for inaccessible $\beta$.
 Let $G_\kappa\subseteq P_\kappa$ be $V$-generic. We claim that the model $V[G_\kappa]$ is as wanted. Let $f_\kappa$ be a $V[G_\kappa]$-generic function for $\Add(\kappa,1)^{V[G_{\kappa}]}$. In $V[G_\kappa][f_\kappa]$ we shell extend $D_p(Club(S^{f_\kappa}_\delta))$ for some $\delta<\kappa$ and $p\in Club(S^{f_\kappa}_\delta)$. First let $U\in V$ be some normal measure on $\kappa$, $$j_1:V\rightarrow M_1\simeq Ult(V,U)$$
 is the corresponding elementary embedding and
 $$j_{1,2}:M_1\rightarrow M_2\simeq Ult (M_1,j_1(U))$$ is the second iteration. Denote by $j_2=j_{1,2}\circ j_1$, $\kappa_i=j_i(\kappa)$ for $i=1,2$.

 Secondly, by the same arguments as in \ref{MasterForcing}, by picking $\Add(\kappa,1)$ at $\name{Q}_\kappa$, we can construct the generic $$\underset{G_{\kappa_1}}{\underbrace{G_{\kappa}*f_\kappa*G_{(\kappa,\kappa_1)}}}*f_{\kappa_1}\in V[G_\kappa][f_\kappa]$$
 which is $M_1$-generic for $j(P_\kappa*\Add(\kappa_1))=P_\kappa*\name{Q}_\kappa*P_{(\kappa,\kappa_1)}*\Add(\kappa_1,1)$. Then the embedding $j_1:V\rightarrow M_1$ lifts to $$j_1^*:V[G_\kappa][f_\kappa]\rightarrow M_1[G_{\kappa_1}][f_{\kappa_1}]$$
 Next, we claim that the forcing $\Add(\beta,1)*Club(S^{f}_\beta)$ is ${<}\beta$-strategically closed when $\beta$ is inaccessible. To see this, let $\lambda<\beta$, Then the good player can always play conditions of the form $\l g,\dot{a}\r\in \Add(\beta,1)*Club(S^{f}_\beta)$ where $\dot{a}$ is the canonical name for some closed set such that $\max(a)=\dom(g)$. The strategy is defined as follows,
 $$\sigma_\lambda(\l\l g_i,\dot{a_i}\r,\l f_i,\name{b_i}\r\mid i<\theta\r)=\langle g,\dot{a}\r$$
 where for limit steps $\theta$, $$g=\bigcup_{i<\theta}g_i\cup\{\langle \nu,\delta\r\}, \ a=\cup_{i<\theta}a_i\cup\{\nu\}$$ $\nu$ being $\sup_{i<\theta}(\sup(\dom(g_i)))$.
 This will form an element of $\Add(\beta,1)*Club(S^{f}_\beta)$ by the definition at successor points $\theta=\tau+1$, in which case $g$ will simply fill the missing points in $\dom(f_\tau)$ with some value different then $\delta$ up to $\sup(\dom(f_\tau))$, if there is a maximal element in  $\dom(f_\tau)$ let $\nu=\max(f_\tau)+1$ otherwise $\nu=\sup(\dom(f_\tau))$ and define $g(\nu)=\delta$. $\name{b_\tau}$ will be extended to a canonical name $\dot{a}$ according to $g$.

 Using this strategically closure of the forcing $$\Add(\kappa_1,1)*\name{Club(S^{f_{\kappa_1}}_\delta)}$$ and the usual arguments of number of dense open sets, in $V[G_\kappa][f_{\kappa}]$ we can find a $M_1[G_{\kappa_1}][f]$-generic club $H$ for $Club(S^{f_{\kappa_1}}_\delta)^{M_1[G_{\kappa_1}][f_\kappa]}$ with $j_1(p)\in H$. Let $C=\cup H\subseteq S^{f_{\kappa_1}}_\delta$ be the generic club.

 Next we shell extend $j_{1,2}:M_1\rightarrow M_2$ to $$j_{1,2}^*:M_1[G_{\kappa_1}][f_{\kappa_1}][H]\rightarrow M_2[G_{\kappa_2}][f_{\kappa_2}][H'']$$
   To do this, note that $$j_{2,1}(P_{\kappa_1}*\Add(\kappa_1,1)*\name{Club(S^{f_{\kappa_1}}_\delta)}=P_{\kappa_1}*\name{Q}_{\kappa_1}*P_{(\kappa_1,\kappa_2)}*\Add(\kappa_2,1)*\name{club(S^{f_{\kappa_2}}_\delta))}$$
   For $P_{\kappa_1}*\name{Q}_\kappa$ we take $G_{\kappa_1}*(f_{\kappa_1}*H)$. For the forcing $P_{(\kappa_1,\kappa_2)}$ we can find a generic $G_{(\kappa_1,\kappa_2)}\in M_1[G_\kappa][f'][H]$ which is $M_2[G_{\kappa_1}][f'][H]$-generic for $P_{(\kappa_1,\kappa_2)}$. Finally, note that the condition $$\langle f_{\kappa_1}\cup\{\l\kappa_1,\delta\r\}, C\cup\{\kappa_1\}\r\in \Add(\kappa_2,1)*\name{Club(S^{f_{\kappa_2}}_\delta))}$$ and once again by the strategically closure and $GCH$ we can extend this condition to a generic $f_{\kappa_2}*H'\in M_1[G_{\kappa_1}][f_{\kappa_1}][H]$.
   So the embedding $j_{1,2}:M_1\rightarrow M_2$ is lifted to
   $$j_{1,2}^*:M_1[G_{\kappa_1}][f_{\kappa_1}][H]\longrightarrow M_2[G_{\kappa_1}][f_{\kappa_1}][H][G_{(\kappa_1,\kappa_2))}][f_{\kappa_2}][H']$$
   By lemma \ref{lemma} there is a condition $m\in H'$ such that for every $q\in H$ $j_U^*(p)\leq m$. In $V[G_\kappa][f]$, define
   $$W=\{x\subseteq Club(S^{f}_\delta)\mid m\in j^*_{2,1}(j^*_1(X)\}$$
   This $\kappa$-complete ultrafilter extends $D_p(Club(S^{f}_\delta))$.
\end{proof}

%starting with a measurable, it is consistent that the filter of dense open subsets of each of the forcing notions $Club(S_\delta)$
%can be extended to a $\kappa$-complete ultrafilter.

\section{Open Problems}
The following question looks natural:

\begin{question}
What is the exact strength of the following assertion: For every $\kappa$-distributive forcing notion of size $\kappa$ the filter of its dense open subsets can be extended to a $\kappa$-complete ultrafilter?
\end{question}
This question is twofold. We can ask what is the \emph{consistency strength} of this assertion and we can also inquire which large cardinals imply it.

%Our conjecture is that it corresponds to $\kappa$ being
% a lightface $\Pi_1^1$-subcompact cardinal.
 %\\Unfortunately, we do not have nice examples of  $\kappa$-distributive forcing notion of size $\kappa$ such that the possibility of extension of the filter of its dense open subsets  to a $\kappa$-complete ultrafilter requires more than just $o(\kappa)=\kappa+\kappa$.
 
 Let $Q$ be the forcing for shooting a club through the singulars. 
\begin{question}
Assume that $D(Q)$ can be extended to a $\kappa$-complete ultrafilter is it consistent that $\exists\lambda.\ o(\lambda)=\lambda^{++}$?
\end{question}
 
 A natural candidate for a forcing for which extending the dense open filter to an ultrafilter might require a higher consistency strength is the forcing of adding a club through a fat stationary set $S\subseteq \kappa$.
 \\However, as it was shown above, depending on the fat stationary set, it may require a measurable alone.
\\
A. Brodsky and A. Rinot \cite{Rinot2019} give a different way to produce many fat stationary sets.
They showed that  $\square(\kappa)$ implies that $\kappa$ can be partitioned into $\kappa$ many disjoint fat stationary sets. In our context, $\kappa$ is a measurable, and so   $\square(\kappa)$ fails.
It is likely that still  in $L[E]$-type models there will be interesting fat sets.

The next question relates to theorem \ref{equivalece}.
Recall that an abstract Prikry type forcing, is a forcing notion $\l \mathcal{Q},\leq,\leq^*\r$ such that $\leq^*\subseteq\leq$, and the Prikry property holds:
$$\text{For every statement in the forcing language }\sigma,\text{ and any condition }q\in\mathcal{Q},$$
$$\text{ there is }q\leq^* q^*\in\mathcal{Q},\text{ such that }q^*\text{ decide }\sigma$$

To obtain interesting Prikry type forcing we usually require that the order $\leq^*$ has high closure or directness degree.

\begin{question}
Is there an abstract generalization of theorem \ref{equivalece} to Prikry type forcing?
Namely, assume there is a projection from a Prikry type forcing $\mathcal{Q}$, for which $\leq^*$ is sufficiently closed or directed onto a distributive forcing $\mathbb{P}$. Can the filter $D_p(\mathbb{P})$ be extended to a $\kappa$-complete ultrafilter?
\end{question}
As we noted after the proof of Theorem \ref{equivalece}, the current formulation does not quite give us an equivalence, as we do not know if the Prikry forcing can be projected onto a distributive forcing notion of size larger than $\kappa$.
\begin{question}
Is there a tree of measures on $\kappa$ such that the corresponding tree Prikry forcing, projects onto a $\sigma$-distributive forcing notion of size $>\kappa$.
\end{question}
\bibliographystyle{amsplain}
\bibliography{ref}
\end{document}